\documentclass[a4paper, reqno]{amsart}

\usepackage{amsmath, amsthm, amssymb}
\usepackage{color}
\usepackage{fixltx2e}
\usepackage{enumitem}
\usepackage{url}
\usepackage{lmodern}

\theoremstyle{plain}
\newtheorem{thm}{Theorem}[section]
\newtheorem{lemma}[thm]{Lemma}

\numberwithin{equation}{section}

\let\originalleft\left
\let\originalright\right
\renewcommand{\left}{\mathopen{}\mathclose\bgroup\originalleft}
\renewcommand{\right}{\aftergroup\egroup\originalright}
\renewcommand\Re{\operatorname{Re}}
\renewcommand\Im{\operatorname{Im}}

\newcommand\twoln[2]{{\substack{#1 \\ #2}}}
\newcommand\thrln[3]{{\substack{#1 \\ #2 \\ #3}}}
\newcommand\e{e}
\newcommand\emod[2]{e\left( \frac{#1}{#2} \right)}

\newcommand\Res[1]{\underset{#1}{\operatorname{Res}}\,\,}
\newcommand\smod[1]{\, \left(#1\right)}

\newcommand\mf[1]{\mathfrak{#1}}

\newcommand\BigO[1]{\mathcal{O}\left(#1\right)}
\newcommand\supp{\operatorname{supp}}
\newcommand\vcent[1]{\vcenter{\hbox{\( \displaystyle #1\)}}}

\title{On a certain additive divisor problem}
\author{Berke Topacogullari}
\address{Mathematisches Institut, Bunsenstr. 3-5, D-37073 G\"ottingen, Germany}
\email{btopaco@uni-goettingen.de}
\thanks{This work was supported by the Volkswagen Foundation}

\begin{document}

\bibliographystyle{abbrv}
\subjclass[2010]{Primary 11N37; Secondary 11N75}
\keywords{divisor function, additive divisor problem, shifted convolution sum, Kuznetsov formula}

\begin{abstract}
  We prove an asymptotic formula for a variant of the binary additive divisor problem with linear factors in the arguments, which has a power saving error term and which is uniform in all involved parameters.
\end{abstract}

\maketitle

\section{Introduction}

Additive divisor problems have a rich history in analytic number theory.
A classical example is given by the problem of finding asymptotic estimates for
\begin{align} \label{eqn: binary additive divisor problem}
  \sum_{n \leq x} d(n) d(n + h), \quad h \geq 1,
\end{align}
also known as the binary additive divisor problem.
There has been a lot of effort in studying this problem (see \cite{Mot94} for a historical survey), one reason being its intimate link to the fourth power moment of the Riemann zeta function.

In other applications to \(L\)-functions (see \cite{BHM07b} and \cite{DFI94a}), variations of this problem have come up, usually stated in the form
\begin{align} \label{eqn: first variation of the binary additive divisor problem}
  D(x_1, x_2) := \sum_{r_1 n_2 - r_2 n_1 = h} w_1\left( \frac{n_1}{x_1} \right) w_2\left( \frac{n_2}{x_2} \right) d(n_1) d(n_2).
\end{align}
Here \(r_1\) and \(r_2\) are positive coprime integers, \(h\) is non-zero, and \(w_1\) and \(w_2\) are smooth weight functions, which we assume to be compactly supported in \( [ 1/2, 1 ] \) (the assumption that \(r_1\) and \(r_2\) be coprime is not restrictive -- otherwise \(h\) has to be divisible by their greatest common divisor, and we can divide both sides of the equation by that number).

Although the classical case \( r_1 = r_2 = 1 \) has probably received most of the attention, there have been some nice results for general \(r_1\), \(r_2\) as well.
Besides the implicit treatment in \cite{BHM07b}, there is the work of Duke, Friedlander and Iwaniec \cite{DFI94b}, who showed that
\begin{align} \label{eqn: result of DFI}
  D(x_1, x_2) = \text{(main term)} + \BigO{ (r_2 x_1 + r_1 x_2)^\frac14 (r_1 r_2 x_1 x_2)^{\frac14 + \varepsilon} }.
\end{align}
As they didn't make use of spectral theory, the size of the error term is inferior compared to what can be achieved for \eqref{eqn: binary additive divisor problem}.
Nevertheless, the range of uniformity in \(r_1\), \(r_2\) and \(h\) for which this asymptotic formula is non-trivial is quite impressive.
At this point we also want to mention the work of Aryan \cite{Ary15}, who improved the result in the case \( r_2 = 1 \).

With applications in mind that will be considered elsewhere, we have come across the following sum, which turned out to be an interesting problem in its own right:
\[ D(x_1, x_2) = \sum_{n} w_1\left( \frac{r_1 n + f_1}{x_1} \right) w_2\left( \frac{r_2 n + f_2}{x_2} \right) d(r_1 n + f_1) d(r_2 n + f_2). \]
Note that for \( (r_1, r_2) = 1 \) and the choice \( h =  r_1 f_2 - r_2 f_1 \), this is exactly the same sum as \eqref{eqn: first variation of the binary additive divisor problem}.
For \(r_1\) and \(r_2\) not coprime, however, we are confronted with a different problem and the following result seems to be new.

\begin{thm} \label{thm: main theorem, smooth version}
  Set
  \begin{align}
    r_0 := \min\left\{ \left( r_1, {r_2}^\infty \right), \left( r_2, {r_1}^\infty \right) \right\} \quad \text{and} \quad h :=  r_1 f_2 - r_2 f_1. \label{eqn: definition of r_0}
  \end{align}
  Then we have for \( h \neq 0 \) and \( f_1 \ll {x_1}^{1 - \varepsilon} \), \( f_2 \ll {x_2}^{1 - \varepsilon} \),
  \[ D(x_1, x_2) = M(x_1, x_2) + \BigO{ r_0 (r_2 x_1)^{\frac12 + \theta + \varepsilon} }, \]
  where the main term is given by
  \[ M(x_1, x_2) := \int \! w_1\left( \frac{r_1 \xi + f_1}{x_1} \right) w_2\left( \frac{r_2 \xi + f_1}{x_2} \right) P( \log(r_1 \xi + f_1), \log(r_2 \xi + f_2) ) \, d\xi, \]
  with \( P(\xi_1, \xi_2) \) a quadratic polynomial depending on \(r_1\), \(r_2\), \(f_1\) and \(f_2\).
  The implicit constants depend only on \(w_1\), \(w_2\) and on \(\varepsilon\).
\end{thm}

By \(\theta\) we denote the bound in the Ramanujan-Petersson conjecture (see section \ref{subsection: Automorphic forms and their Fourier expansions} for a precise definition).
The polynomial in the theorem above can be stated fairly explicitly, see \eqref{eqn: the polynomial in the main term}.
We haven't aimed for the largest possible range of uniformity in \(f_1\) and \(f_2\).
In fact, with some work it should be possible to extend our result to include \(f_1\), \(f_2\) in a larger range than required above.
It also seems likely that the dependance on \(r_0\) is not optimal, but here it is not immediately clear how an improvement might be achieved.
Compared to \eqref{eqn: result of DFI} our result has a better error term, although their estimate is valid for much larger \(h\) than ours.
In the case \( r_2 = 1 \), our result is the same as \cite[Theorem 0.3]{Ary15}.

We also state the following analogous result for the sum with sharp cut-off.

\begin{thm} \label{thm: main theorem}
  Let \(r_0\) and \( h \neq 0 \) be defined as above.
  Assume that
  \[ f_1 \ll (r_1 x)^{1 - \varepsilon}, \quad f_2 \ll (r_2 x)^{1 - \varepsilon} \quad \text{and} \quad (r_0 r_1 r_2, h) h \ll {r_0}^\frac13 (r_1 r_2)^\frac53 x^{\frac13 - \varepsilon}. \]
  Then
  \[ \sum_{\frac x2 < n \leq x} d(r_1 n + f_1) d(r_2 n + f_2) = x P(\log x) + \BigO{ (r_0 r_1 r_2, h)^\theta {r_0}^{\frac23 + \theta} (r_1 r_2)^\frac13 x^{\frac23 + \varepsilon} }, \]
  where \( P(\xi) \) is a quadratic polynomial depending on \(r_1\), \(r_2\), \(f_1\) and \(f_2\), and where the implicit constants only depend on \(\varepsilon\).
\end{thm}

Correlations of a much more general type have been investigated by Matthiesen \cite{Mat12}, but the methods used there don't apply to our case and don't give power savings in the error term.
Similar problems, where the divisor functions are replaced by Fourier coefficients of automorphic forms, have been studied as well (see e.g. \cite{Blo04}).
In particular, for Fourier coefficients of holomorphic cusp forms, Pitt \cite[Theorem 1.4]{Pitt13} was able to prove an analogue of our Theorem \ref{thm: main theorem, smooth version} for \(r_1\), \(r_2\) squarefree and \( f_1 = f_2 = -1 \).
Unfortunately, his method relies on Jutila's variant of the circle method and is not applicable to our case.

The proof of Theorems \ref{thm: main theorem, smooth version} and \ref{thm: main theorem} follows standard lines:
We split one of the divisor functions and use the Voronoi summation formula to deal with the divisor sum in arithmetic progressions.
The main difficulty lies in the handling of the sum of Kloosterman sums entering the stage at this point.
In a simplified form, we are faced with a sum roughly of the shape
\[ \sum_\twoln{c}{ (c, r_2) = 1 } \frac{ S(1 - r_1 \overline{r_2}, 1; r_1 c) }{r_1c} F(r_1c), \]
where \( F \) is some weight function, and where \( \overline{r_2} \) is understood to be mod \(c\).
We could bound the Kloosterman sums individually using Weil's bound, and the resulting error terms in our theorems would be of a size comparable to \eqref{eqn: result of DFI}.
Our aim however is to use spectral methods to get results beyond that.

If \(r_1\) and \(r_2\) are coprime, we can use the Kuznetsov formula with an appropriate choice of cusps to do that.
Otherwise, it is not directly clear how the Kuznetsov formula might be put into use here.
In this article we want to show that nevertheless this is possible.
We solve the problem by splitting the variable \( r_1 = tv \) into a factor \(t\), which is coprime to \(r_2\), and a factor \(v\), which contains only the same prime factors as \(r_2\).
By twisted multiplicativity of Kloosterman sums we have
\[ \frac{ S(1 - r_1 \overline{r_2}, 1; r_1c) }{r_1 c} = \frac{ S( \overline{tc}, \overline{tc}; v ) }v \frac{ S\left( \vcent{ r_2 - r_1, \overline{ v^2 r_2}; t c } \right) }{t c}, \]
where now all the inverses are understood to be modulo the respective modulus of the Kloosterman sum.
Following an idea of Blomer and Mili\'{c}evi\'{c} \cite{BM15}, we separate the variable \(c\) occuring in the first factor by exploiting the orthogonality of Dirchlet characters, namely as follows
\[ \frac{ S( \overline{t c}, \overline{t c}; v ) }v = \frac1{ \varphi(v) } \sum_{ \chi \bmod v } \chi(t c) \hat S_v(\chi), \quad \text{with} \quad \hat S_v(\chi) := \sum_\twoln{y \smod v}{ (y, v) = 1 } \overline\chi(y) \frac{ S( \overline y, \overline y; v ) }v, \]
where the left sum runs over all Dirichlet characters mod \(v\).
This way we are led to a sum of Kloosterman sums twisted by a Dirichlet character, which we can treat by spectral methods.

\section{Preliminaries}

Note that \( \varepsilon \) always stands for some positive real number, which can be chosen arbitrarily small.
However, it need not be the same on every occurrence, even if it appears in the same equation.
To avoid confusion we also want to recall that as usual \( e(z) := e^{2\pi i z} \), and that
\[ S(m, n; c) := \sum_\twoln{a \smod c}{ (a, c) = 1 } \emod{ ma + n \overline{a} }c \quad \text{and} \quad c_q(n) := \sum_\twoln{a \smod q}{ (a, q) = 1 } \emod{an}q, \]
which are the usual notations for Kloosterman sums and Ramanujan sums.

\subsection{The Voronoi summation formula and Bessel functions}

Using the well-known Voronoi formula for the divisor function (see \cite[Chapter 4.5]{IK04} or \cite[Theorem 1.6]{Jut87}) and the identity
\[ \sum_\twoln{n = 1}{ n \equiv b \smod c }^\infty d(n) f(n) = \frac1c \sum_{d \mid c} \sum_\twoln{\ell \smod d}{ (\ell, d) = 1 } \emod{-b\ell}d \sum_{n = 1}^\infty d(n) f(n) \emod{n\ell}d, \]
it is not hard to show the following summation formula for the divisor function in arithmetic progressions:

\begin{thm}\label{thm: Voronoi summation for d(n) in arithmetic progressions}
  Let \(b\) and \( c \geq 1 \) be integers.
  Let \( f : (0, \infty) \rightarrow \mathbb{R} \) be smooth and compactly supported.
  Then
  \begin{align*}
    \sum_{ n \equiv b \smod c } d(n) f(n) &= \frac1c \int \! \lambda_{b, c}(\xi) f(\xi) \, d\xi \\
      &\qquad -\frac{2 \pi}c \sum_{d \mid c} \sum_{n = 1}^\infty d(n) \frac{ S(b, n; d) }d \int \! Y_0\left( \frac{4 \pi}d \sqrt{n \xi} \right) f(\xi) \, d\xi \\
      &\qquad +\frac4c \sum_{d \mid c} \sum_{n = 1}^\infty d(n) \frac{ S(b, -n; d) }d \int \! K_0\left( \frac{4 \pi}d \sqrt{n \xi} \right) f(\xi) \, d\xi,
  \end{align*}
  with
  \[ \lambda_{b, c}(\xi) := \sum_{d \mid c} \frac{ c_d(b) }d (\log\xi + 2 \gamma - 2 \log d). \]
\end{thm}

Concerning the Bessel functions appearing in the above Theorem, we want to sum up some well-known facts.
We know that
\begin{align} \label{eqn: bounds for the K-Bessel function}
  K_0(\xi) \ll | \log\xi | \quad \text{for} \quad \xi \ll 1, \quad \text{and} \quad K_0 \ll \frac1{ e^\xi \sqrt{\xi} } \quad \text{for} \quad \xi \gg 1,
\end{align}
and that for \( \mu \geq 1 \),
\[ K_0^{ (\mu) }(\xi) \ll \frac1{ \xi^\mu } \quad \text{for} \quad \xi \ll 1, \quad \text{and} \quad K_0^{ (\mu) } \ll \frac1{ e^\xi \sqrt{\xi} } \quad \text{for} \quad \xi \gg 1. \]
Regarding the \(Y\)-Bessel function, we have for \( \nu \geq 1 \) and \( \xi \ll 1 \),
\[ Y_0(\xi) \ll | \log\xi |, \quad Y_\nu(\xi) \ll \frac1{ \xi^\nu }, \quad \text{and} \quad Y_0^{ (\mu) } \ll \frac1{ \xi^\mu }, \quad Y_\nu^{ (\mu) } \ll \frac1{ \xi^{\nu + \mu} } \quad \text{for} \quad \mu \geq 1. \]
For \( \nu \geq 0 \) and \( \xi \gg 1 \), it is known that
\[ Y_\nu^{ (\mu) }(\xi) \ll \frac1{ \sqrt{\xi} } \quad \text{for} \quad \mu \geq 0. \]

From the recurrence relation
\[ \left( \xi^\nu Y_\nu(\xi) \right)' = \xi^\nu Y_{\nu - 1}(\xi), \]
we get the identity
\begin{align}\label{eqn: consequence of the recurrence relations}
  \int \! Y_0\left( \frac{4\pi}c \sqrt{h \xi} \right) f(\xi) \, d\xi = \left( \frac{-2c}{ 4\pi \sqrt{h} } \right)^\nu \int \! \xi^\frac\nu2 Y_\nu\left( \frac{4\pi}c \sqrt{h \xi} \right) \frac{\partial^\nu f}{\partial \xi^\nu}(\xi) \, d\xi,
\end{align}
which is useful when estimating the sizes of the Bessel transforms occuring in the Voronoi summation formula.

The \(Y_\nu\)-Bessel functions oscillate for large values, and to make use of this behaviour we state the following lemma.

\begin{lemma} \label{lemma: Bessel functions as oscillating functions}
  For any \( \nu \geq 0 \) there is a smooth function \( v_Y: (0, \infty) \rightarrow \mathbb{C} \) such that
  \[ Y_\nu(\xi) = 2\Re\left( \e\left( \frac \xi{2\pi} \right) v_Y\left( \frac \xi\pi \right) \right), \]
  and such that for any \( \mu \geq 0 \),
  \[ v_Y^{ (\mu) } \ll \frac1{ \xi^{\mu + \frac12} } \quad \text{for} \quad \xi \gg 1. \]
\end{lemma}

\begin{proof}
  This can be shown by using the integral representation (see \cite[3.871]{GR07})
  \[ Y_0(\xi) = -\frac1\pi \int_0^\infty \! \cos\left( \frac x{2\pi} + \frac{\pi \xi^2}{2x} \right) \, \frac{dx}x, \]
  and applying the same variable substitution as in \cite[Lemma 4]{DI82b}.
  See \cite[Lemma 2.3]{Top15} for more details.
\end{proof}

\subsection{The Hecke congruence subgroup and Kloosterman sums}\label{subsection: The Kuznetsov trace formula and the Large sieve inequalities}

Here and in the following sections we will go through some results from the theory of automorphic forms.
A general description of the spectral theory of automorphic forms can be found for instance in \cite{Iwa02} or \cite[Chapters 14--16]{IK04}, while \cite{DFI02} gives a very nice introduction to Maa{\ss} forms of higher weight with arbitrary nebentypus.

Besides the Kuznetsov trace formula, our main tools are the large sieve inequalities, which were proven by Deshouillers and Iwaniec \cite{DI82a} with respect to Hecke congruence subgroups.
Their results can be extended to our specific setting, the details of which have luckily been worked out by Drappeau \cite{Dra15}.
Finally, we also want to cite \cite{BHM07b} as a reference, where we borrow large parts of the notation.

Let \(q\) be some positive integer, let \( \kappa \in \{0, 1\} \), and let \(\chi\) be a character mod \(q_0\), with \( q_0 \mid q \), such that
\[ \chi(-1) = (-1)^\kappa. \]
Let \( \Gamma := \Gamma_0(q) \) be the Hecke congruence subgroup of level \(q\).
The character \(\chi\) naturally extends to \(\Gamma\) by setting
\[ \chi(\gamma) := \chi(d) \quad \text{for} \quad \begin{pmatrix} a & b \\ c & d \end{pmatrix} \in \Gamma. \]
Every cusp \( \mathfrak{a} \) of \(\Gamma\) is equivalent to some \( \frac uw \) with \( (u, w) = 1 \) and \( w \mid q \).
It is called singular if
\[ \chi(\gamma) = 1 \quad \text{for all } \gamma \in \Gamma_\mathfrak{a}, \]
where \( \Gamma_\mathfrak{a} \) is the stabilizer of \( \mathfrak{a} \).

For any cusp \( \mathfrak{a} \) of \(\Gamma\) we can choose \( \sigma_\mathfrak{a} \in \operatorname{SL}_2( \mathbb{R} ) \) such that
\[ \sigma_\mathfrak{a} \infty = \mathfrak{a} \quad \text{and} \quad { \sigma_\mathfrak{a} }^{-1} \Gamma_\mathfrak{a} \sigma_\mathfrak{a} = \Gamma_\infty. \]
Given two singular cusps \( \mathfrak{a} \), \( \mathfrak{b} \), we define for \( n, m \in \mathbb{Z} \) the Kloosterman sum
\[ S_\mathfrak{ab}(m, n; \gamma) := \sum_{ \delta \bmod \gamma \mathbb{Z} } \overline\chi\left( \sigma_\mathfrak{a} \begin{pmatrix} \alpha & \beta \\ \gamma & \delta \end{pmatrix} { \sigma_\mathfrak{b} }^{-1} \right) \e\left( m \frac\alpha\gamma + n \frac\delta\gamma \right), \]
where the sum runs over all \( \delta \bmod \gamma\mathbb{Z} \), for which there exist some \(\alpha\), \(\beta\) such that
\[ \begin{pmatrix} \alpha & \beta \\ \gamma & \delta \end{pmatrix} \in { \sigma_\mathfrak{a} }^{-1} \Gamma \sigma_\mathfrak{b}. \]
Note that this definition depends on the chosen scaling matrices \( \sigma_\mathfrak{a} \) and \( \sigma_\mathfrak{b} \).

As an example, for \( \mathfrak{a} = \mathfrak{b} = \infty \) and the choice \( \sigma_\infty = 1 \), the sum is non-empty exactly when \( q \mid c \) and in this case it reduces to the usual twisted Kloosterman sum
\[ S_{\infty \infty}(m, n; c) = S_\chi(m, n; c) := \sum_\twoln{a \smod c}{ (a, c) = 1 } \chi(a) \emod{ ma + n \overline a }c. \]
It is well-known that for any prime \(p\) this sum can be bound by
\[ S_\chi(m, n; p) \leq 2 (m, n, p)^\frac12 p^\frac12. \]
However, for general modulus we have to account for the conductor of \(\chi\) as well, and in this case the following bound holds (see \cite[Theorem 9.2]{KL13})
\[ S_\chi(m, n; c) \ll (m, n, c)^\frac12 {q_0}^\frac12 c^{\frac12 + \varepsilon}. \]

Another important example is given for \(q\) having the form \( q = rs \) with \( (r, s) = 1 \) and \( q_0 \mid r \).
Consider the two singular cusps \(\infty\) and \( \frac1s\), together with the choice
\[ \sigma_\frac1s = \begin{pmatrix} \sqrt{r} & 1 \\ s \sqrt{r} & \sqrt{r}^{-1} \end{pmatrix}. \]
Now the sum \( S_{\infty \frac1s}(m, n; \gamma) \) is non-empty exactly when \(\gamma\) may be written as
\[ \gamma = \sqrt{r} s c, \quad \text{with} \quad c \in \mathbb{Z} \setminus \{0\}, \quad (c, r) = 1, \]
and in this case we have
\[ S_{\infty \frac1s}(m, n; \gamma) = \e\left( n \frac{ \overline{s} }r \right) \overline\chi(c) S\left( m, n \overline r; sc \right). \]

\subsection{Automorphic forms and their Fourier expansions} \label{subsection: Automorphic forms and their Fourier expansions}

By \( \mathcal{S}_k(q, \chi) \) we denote the finite-dimensional Hilbert space of holomorphic cusp forms of weight \( k \equiv \kappa \bmod 2 \) with respect to \( \Gamma_0(q) \) and with nebentypus \(\chi\).
Let \( \theta_k(q, \chi) \) be its dimension.
For each \(k\), we choose an orthonormal Hecke eigenbasis \( f_{j, k} \), \( 1 \leq j \leq \theta_k(q, \chi) \).
Then the Fourier expansion of \( f_{j, k} \) around a singular cusp \( \mathfrak{a} \) (with associated scaling matrix \( \sigma_\mathfrak{a} \)) is given by
\[ i(\sigma_\mathfrak{a}, z)^{-k} f_{j, k}(\sigma_\mathfrak{a} z) = \sum_{n = 1}^\infty \psi_{j, k}(n, \mathfrak{a}) (4\pi n)^\frac k2 e(nz), \]
where we have set
\[ i(\gamma, z) := cz + d \quad \text{for} \quad \gamma = \begin{pmatrix} a & b \\ c & d \end{pmatrix}. \]

Next, let \( \mathcal{L}^2(q, \chi) \) be the space of Maa{\ss} forms of weight \(\kappa\) with respect to \( \Gamma_0(q) \) and with nebentypus \(\chi\), and let \( \mathcal{L}_0^2(q, \chi) \subset \mathcal{L}^2(q, \chi) \) be its subspace of Maa{\ss} cusp forms.
Let \( u_j \), \( j \geq 1 \), run over an orthonormal Hecke eigenbasis of \( \mathcal{L}_0^2(q, \chi) \), with the corresponding real eigenvalues \( \lambda_1 \leq \lambda_2 \leq \ldots \).
We can assume each \(u_j\) to be either even or odd.
We set \( {t_j}^2 = \lambda_j - \frac14 \), where we choose the sign of \( t_j \) so that \( i t_j \geq 0 \) if \( \lambda_j < \frac14 \), and \( t_j \geq 0 \) if \( \lambda_j \geq \frac14 \).
Then the Fourier expansions of these functions around a singular cusp \( \mathfrak{a} \) is given by
\[ j(\sigma_\mathfrak{a}, z)^{-\kappa} u_j(\sigma_\mathfrak{a} z) = \sum_{n \neq 0} \rho_j( n, \mathfrak{a} ) W_{ \frac n{ |n| } \frac\kappa2, i t_j  }(4 \pi |n| y) e(nx), \]
where
\[ j(\gamma, z) := \frac{ c z + d }{ |cz + d| } \quad \text{for} \quad \gamma = \begin{pmatrix} a & b \\ c & d \end{pmatrix}. \]
The Selberg eigenvalue conjecture says that \( \lambda_1 \geq \frac14 \), which would imply that all \( t_j \) are real and non-negative.
While for \( \kappa = 1 \) this is known to be true, for \( \kappa = 0 \) it is still an open question.
The eigenvalues with \( 0 < \lambda_j < \frac14 \) as well as the corresponding values \( t_j \) are called exceptional, and lower bounds for these exceptional \( \lambda_j \) imply upper bounds for the corresponding \( i t_j \).
Let \( \theta \in [0, \infty) \) be such that \( i t_j \leq \theta \) for all exceptional \( t_j \) uniformly for all levels \(q\) and any nebentypus; by the work of Kim and Sarnak \cite{Kim03} we know that we can choose
\[ \theta = \frac7{64}. \]

The orthogonal complement to \( \mathcal{L}_0^2(q, \chi) \) in \( \mathcal{L}^2(q, \chi) \) is the Eisenstein spectrum \( \mathcal{E}(q, \chi) \) (plus possibly the space of constant functions if \(\chi\) is trivial).
It can be described explicitly by means of the Eisenstein series \( E_\mathfrak{c}\left( z; \frac12 + it \right) \), where \( \mathfrak{c} \) is a singular cusp and \( t \in \mathbb{R} \).
The Fourier expansion of these Eisenstein series around the cusp \( \mathfrak{a} \) is given by
\begin{multline*}
  j(\sigma_\mathfrak{a}, z)^{-\kappa} E_\mathfrak{c}\left(\sigma_\mathfrak{a} z; \frac12 + it \right) = c_{ \mathfrak{c}, 1 }(t) y^{\frac12 + it} + c_{ \mathfrak{c}, 2 }(t) y^{\frac12 - it} \\
    + \sum_{n \neq 0} \varphi_{ \mathfrak{c}, t }( n, \mathfrak{a} ) W_{ \frac n{ |n| } \frac\kappa2, it  }(4 \pi |n| y) e(nx).
\end{multline*}

Note that by the choice of our basis, we have that
\[ | \rho_j(-n, \infty) | = |t_j|^\kappa | \rho_j(n, \infty) | \quad \text{for} \quad n \geq 1. \]
Furthermore, since all Eisenstein series are even, the same is true for their Fourier coefficients, namely
\[ | \varphi_{\mf c, t}(-n, \infty) | = |t|^\kappa | \varphi_{\mf c, t}(n, \infty) | \quad \text{for} \quad n \geq 1. \]

\subsection{The Kuznetsov trace formula}

With the whole notation set up, we can now formulate the famous Kuznetsov trace formula, which in our case reads as follows.

\begin{thm}\label{thm: Kuznetsov trace formula}
  Let \( f : (0, \infty) \rightarrow \mathbb{C} \) be smooth with compact support, let \( \mathfrak{a} \), \( \mathfrak{b} \) be singular cusps, and let \(m\), \(n\) be positive integers.
  Then
  \begin{align*}
      \sum_\gamma \frac{ S_\mathfrak{ab}(m, n; \gamma) }\gamma &f\left( 4\pi \frac{ \sqrt{mn} }\gamma \right) = \sum_{j = 1}^\infty \overline{ \rho_j }(m, \mathfrak{a}) \rho_j(n, \mathfrak{b}) \frac{ \sqrt{mn} }{ \cosh(\pi t_j) } \tilde f( t_j ) \\
        &+ \sum_{ \mathfrak{c} \text{ sing.} } \frac1{4\pi} \int_{-\infty}^\infty \! \overline{ \varphi_{ \mathfrak{c}, t } } \left( m, \mathfrak{a} \right) \varphi_{\mathfrak{c}, t } \left( n, \mathfrak{b} \right) \frac{ \sqrt{mn} }{ \cosh(\pi t) } \tilde f(t) \, dt \\
        &+ \sum_\twoln{ k \equiv \kappa \smod 2, \,\, k > \kappa }{ 1 \leq j \leq \theta_k(q, \chi) } (k - 1)! \, \overline{ \psi_{j, k} }(m, \mathfrak{a}) \psi_{j, k}(n, \mathfrak{b}) \sqrt{mn} \dot f(k),
    \intertext{and}
      \sum_\gamma \frac{ S_\mathfrak{ab}(m, -n; \gamma) }\gamma &f\left( 4\pi \frac{ \sqrt{mn} }\gamma \right) = \sum_{j = 1}^\infty \overline{ \rho_j }(m, \mathfrak{a}) \rho_j(-n, \mathfrak{b}) \frac{ \sqrt{mn} }{ \cosh(\pi t_j) } \check f( t_j ) \\
        &+ \sum_{ \mathfrak{c} \text{ sing.} } \frac1{4\pi} \int_{-\infty}^\infty \! \overline{ \varphi_{ \mathfrak{c}, t } } \left( m, \mathfrak{a} \right) \varphi_{\mathfrak{c}, t } \left( -n, \mathfrak{b} \right) \frac{ \sqrt{mn} }{ \cosh(\pi t) } \check f(t) \, dt,
  \end{align*}
  where \(\gamma\) runs over all positive real numbers for which \( S_{ \mathfrak{a} \mathfrak{b} }(m, n; \gamma) \) is non-empty, and where the Bessel transforms are defined by
  \begin{align*}
    \tilde f(t) &= \frac{2\pi i t^\kappa}{ \sinh(\pi t) } \int_0^\infty \! \left( J_{2it}(\xi) - (-1)^\kappa J_{-2it}(\xi) \right) f(\xi) \, \frac{d\xi}\xi, \\
    \check f(t) &= 8 i^{-\kappa} \cosh(\pi t) \int_0^\infty \! K_{2it}(\xi) f(\xi) \, \frac{d\xi}\xi, \\
    \dot f(k) &= 4 i^k \int_0^\infty \! J_{k - 1}(\xi) f(\xi) \, \frac{d\xi}\xi.
  \end{align*}
\end{thm}

\begin{proof}
  The proof of these formulas can be done along standard lines, as described for instance in \cite[chapter 16.4]{IK04} (see also \cite{Pro03}).
  The extension to our setting and to general cusps poses no real problems.
  
  We just want to point out \cite[Proposition 5.2]{DFI02}, which can be used as a starting point for the proof.
  Its analogue for the case of mixed signs follows by a slight modification of the argument described there, and is given by
  \begin{multline*}
    \sum_{j = 1}^\infty \frac{ \overline{ \rho_j }(m, \mathfrak{a}) \rho_j(-n, \mathfrak{b}) \sqrt{mn} }{ \cosh( \pi(r - t_j) ) \cosh( \pi(r + t_j) ) } \\
      + \sum_{ \mathfrak{c} \text{ sing.} } \frac1{4\pi} \int_{-\infty}^\infty \! \frac{ \overline{ \varphi_{ \mathfrak{c}, t } } \left( m, \mathfrak{a} \right) \varphi_{\mathfrak{c}, t } \left( -n, \mathfrak{b} \right) \sqrt{mn} }{ \cosh( \pi(r - t_j) ) \cosh( \pi(r + t_j) )  } \, dt \\
      = \frac{ 2 \sqrt{mn} }{ \pi i^\kappa } \sum_\gamma \frac{ S_\mathfrak{ab}(m, -n; \gamma) }{\gamma^2} K_{2ir}\left( 4\pi \frac{ \sqrt{mn} }\gamma \right),
  \end{multline*}
  being true for positive integers \(m\), \(n\) and \( r \in \mathbb{R} \).
\end{proof}

In case \( \mathfrak{a} = \mathfrak{b} = \infty \), the sum of Kloosterman sums in the theorem above is just
\begin{align}
  \sum_\gamma \frac{ S_{\infty \infty}(m, \pm n; \gamma) }\gamma f\left( 4\pi \frac{ \sqrt{mn} }\gamma \right) &= \sum_{ c \equiv 0 \smod q } \frac{ S_\chi(m, \pm n; c) }c f\left( 4\pi \frac{ \sqrt{mn} }c \right), \nonumber
\end{align}
while in the case \( q = rs \) with \( (r, s) = 1 \) and \( q_0 \mid r \) mentioned above, we have
\begin{multline}
  \sum_\gamma \frac{ S_{\infty \frac1s}(m, \pm n; \gamma) }\gamma f\left( \frac{ 4\pi \sqrt{mn} }\gamma \right) \\
    = \e\left( \frac{ \pm n \overline{s} }r  \right) \sum_\twoln{c}{ (c, r) = 1 } \overline \chi(c) \frac{ S\left( m, \pm n \overline r; sc \right) }{ \sqrt{r} sc } f\left( \frac{ 4\pi \sqrt{mn} }{ \sqrt{r} sc } \right). \label{eqn: Kuznetsov formula for q = rs}
\end{multline}
To get some first estimates for the Bessel transforms appearing above we refer to \cite[Lemma 2.1]{BHM07a}, where the case \( \kappa = 0 \) is covered.
The proofs carry over to the case \( \kappa = 1 \) with minimal changes.

\begin{lemma}\label{lemma: estimates for the Bessel transforms}
  Let \( f : (0, \infty) \rightarrow \mathbb{C} \) be a smooth and compactly supported function such that
  \[ \supp f \asymp X \quad \text{and} \quad f^{ (\nu) }(\xi) \ll \frac1{Y^\nu} \quad \text{for} \quad \nu = 0, 1, 2, \]
  for positive \(X\) and \(Y\) with \( X \gg Y \).
  Then
  \begin{align*}
    \tilde f(it), \check f(it) &\ll \frac{ 1 + Y^{-2t} }{1 + Y} & \text{for} \quad 0 \leq t < \frac14, \\
    \frac{ \tilde f(t) }{ (1 + t)^\kappa }, \check f(t), \dot f(t) &\ll \frac{ 1 + |\log Y| }{1 + Y} & \text{for} \quad t \geq 0, \\
    \frac{ \tilde f(t) }{ (1 + t)^\kappa }, \check f(t), \dot f(t) &\ll \left( \frac XY \right)^2 \left( \frac1{ t^\frac52 } + \frac X{ t^3 } \right) & \text{for} \quad t \gg \max(X, 1).
  \end{align*}
\end{lemma}

For oscillating functions, we can do better.
Assume \( w : (0, \infty) \rightarrow \mathbb{C} \) to be a smooth and compactly supported function such that
\[ \supp w \asymp X \quad \text{and} \quad w^{ (\nu) }(\xi) \ll \frac1{X^\nu} \quad \text{for} \quad \nu \geq 0, \]
and for \( \alpha > 0 \) define
\[ f(\xi) := \e\left( \xi \frac\alpha{2 \pi} \right) w(\xi). \]
Then we have the following bounds.

\begin{lemma}\label{lemma: estimates for the Bessel transforms of oscillating functions for large alpha}
  Assume that
  \[ X \ll 1 \quad \text{and} \quad \alpha X \gg 1. \]
  Then for \( \nu, \mu \geq 0 \),
  \begin{align}
    \tilde f(it), \check f(it) &\ll X^{ -2t + \varepsilon } \left( X^\mu + \frac1{ (\alpha X)^\nu } \right) &\text{for} \quad 0 < t \leq \frac14, \label{eqn: first estimate for the bessel transforms for osciallting functions} \\
    \frac{ \tilde f(t) }{ (1 + t)^\kappa }, \check f(t), \dot f(t) &\ll \frac{\alpha^\varepsilon}{\alpha X} \left( \frac{\alpha X}t \right)^\nu &\text{for} \quad t > 0. \label{eqn: second estimate for the bessel transforms for osciallting functions}
  \end{align}
\end{lemma}

\begin{proof}
  The bound \eqref{eqn: first estimate for the bessel transforms for osciallting functions} can be shown by making use of the Taylor series of the respective Bessel functions.
  The proof of \eqref{eqn: second estimate for the bessel transforms for osciallting functions} is a variation of the proof of \cite[Lemma 3]{Jut99}.
  See \cite[Lemma 2.6]{Top15} for details.
\end{proof}

\subsection{Large sieve inequalities and estimates for Fourier coefficients}

Another important tool are the large sieve inequalities for Fourier coefficients of cusp forms and Eisenstein series.
For a sequence \( a_n \) of complex numbers define
\[ \| a_n \|_N := \sqrt{ \sum_{N < n \leq 2N} |a_n|^2 }, \]
and furthermore set

\begin{align*}
  \Sigma_{j, \pm}^{(1)} (N) &:= \frac{ ( 1 + |t_j| )^{\pm \frac\kappa2} }{ \sqrt{ \cosh(\pi t_j) } } \sum_{N < n \leq 2N} a_n \rho_j( \pm n, \mathfrak{a} ) \sqrt{n}, \\
  \Sigma_{ \mathfrak{c}, t, \pm }^{(2)} (N) &:= \frac{ ( 1 + |t| )^{\pm \frac\kappa2} }{ \sqrt{ \cosh(\pi t) } } \sum_{N < n \leq 2N} a_n \varphi_{ \mathfrak{c}, t }( \pm n, \mathfrak{a} ) \sqrt{n}, \\
  \Sigma_{j, k}^{(3)} (N) &:= \sqrt{ (k - 1)! } \sum_{N < n \leq 2N} a_n \psi_{j, k}( n, \mathfrak{a} ) \sqrt{n}.
\end{align*}

Then the following bounds are known as the large sieve inequalities.

\begin{thm}\label{thm: large sieve inequalities}
  Let \( T \geq 1 \) and \( N \geq \frac12 \) be real numbers, \( a_n \) a sequence of complex numbers, and \( \mathfrak{a} \) a singular cusp of \(\Gamma\) written in the form \( \mathfrak{a} = \frac uw \) with \( (u, w) = 1 \).
  Then
  \begin{align*}
    \sum_{ | t_j | \leq T } \left| \Sigma_{j, \pm}^{(1)} (N) \right|^2 &\ll \left( T^2 + {q_0}^\frac12 \left( w, \frac qw \right) \frac{ N^{1 + \varepsilon} }q \right) \| a_n \|_N^2, \\
    \sum_{ \mathfrak{c} \text{ sing.} } \int_{-T}^T \! \left| \Sigma_{ \mathfrak{c}, t, \pm }^{(2)} (N) \right|^2 \, dt &\ll \left( T^2 + {q_0}^\frac12 \left( w, \frac qw \right) \frac{ N^{1 + \varepsilon} }q \right) \| a_n \|_N^2, \\
    \sum_\twoln{ k \leq T, \,\, k \equiv \kappa \smod 2 }{ 1 \leq j \leq \theta_k(q, \chi) } \left| \Sigma_{k, j}^{(3)}(N) \right|^2 &\ll \left( T^2 + {q_0}^\frac12 \left( w, \frac qw \right) \frac{ N^{1 + \varepsilon} }q \right) \| a_n \|_N^2,
  \end{align*}
  where the implicit constants depend only on \( \varepsilon \).
\end{thm}

\begin{proof}
  With the appropriate changes, these bounds can be deduced essentially in the same way as it is done in \cite[section 5]{DI82a}.
  We refer to \cite{Dra15} for details.
\end{proof}

When there is no averaging over \(n\), the following lemma gives useful bounds, especially when \(q\) or \(T\) is large.

\begin{lemma}\label{lemma: estimates for Fourier coefficients, I}
  Let \( T \geq 1 \), \( n \geq 1 \), and \(\mf a\) as above.
  Then
  \begin{align*}
    \sum_{ | t_j | \leq T } \frac{ ( 1 + |t_j| )^{\pm \kappa} }{ \cosh(\pi t_j) } | \rho_j( \pm n, \mathfrak{a} ) |^2 n &\ll T^2 + (qnT)^\varepsilon (q, n)^\frac12 {q_0}^\frac12 \left( w, \frac qw \right) \frac{ n^\frac12 }q, \\
    \sum_{ \mathfrak{c} \text{ sing.} } \int_{-T}^T \! \frac{ ( 1 + |t| )^{\pm \kappa} }{ \cosh(\pi t) } | \varphi_{ \mathfrak{c}, t }( \pm n, \mathfrak{a} ) |^2 n \, dt &\ll T^2 + (qnT)^\varepsilon (q, n)^\frac12 {q_0}^\frac12 \left( w, \frac qw \right) \frac{ n^\frac12 }q, \\
    \sum_\twoln{ k \leq T, \,\, k \equiv \kappa \smod 2 }{ 1 \leq j \leq \theta_k(q, \chi) } (k - 1)! \left| \psi_{j, k}( n, \mathfrak{a} ) \right|^2 n &\ll T^2 + (qnT)^\varepsilon (q, n)^\frac12 {q_0}^\frac12 \left( w, \frac qw \right) \frac{ n^\frac12 }q,
  \end{align*}
  where the implicit constants depend only on \( \varepsilon \).
\end{lemma}

\begin{proof}
  For the full modular group and trivial nebentypus, a proof for the first two bounds can be found for example in \cite[Lemma 2.4]{Mot97}.
  Using an appropriate trace formula as starting point (e.g. \cite[Proposition 5.2]{DFI02}) , the proof carries over easily to our case.
  Except for the same kind of modifications, the proof of the last bound is a simpler variant of \cite[Proposition 4]{DI82a}.
\end{proof}

For \(n\) large, the following bounds are often better.

\begin{lemma}\label{lemma: estimates for Fourier coefficients, II}
  Let \( T \geq 1 \), \( n \geq 1 \) and \(\mf a\) as above.
  Then
  \begin{align}
    \sum_{ | t_j | \leq T } \frac{ ( 1 + |t_j| )^{\pm \kappa} }{ \cosh(\pi t_j) } | \rho_j( \pm n, \infty ) |^2 n &\ll (qnT)^\varepsilon T^2 n^{2\theta}, \label{eqn: estimate for the Fourier coefficients of Maass cusp forms} \\
    \sum_{ \mathfrak{c} \text{ sing.} } \int_{-T}^T \! \frac{ ( 1 + |t| )^{\pm \kappa} }{ \cosh(\pi t) } | \varphi_{ \mathfrak{c}, t }( \pm n, \infty ) |^2 n \, dt &\ll (qnT)^\varepsilon T,  \label{eqn: estimate for the Fourier coefficients of Eisenstein series} \\
    \sum_\twoln{ k \leq T, \,\, k \equiv \kappa \smod 2 }{ 1 \leq j \leq \theta_k(q, \chi) } (k - 1)! \left| \psi_{j, k}( n, \infty ) \right|^2 n &\ll (qnT)^\varepsilon T^2 , \label{eqn: estimate for the Fourier coefficients of holomorphic cusp forms}
  \end{align}
  where the implicit constants depend only on \( \varepsilon \).
\end{lemma}

\begin{proof}
  The bounds \eqref{eqn: estimate for the Fourier coefficients of Maass cusp forms} and \eqref{eqn: estimate for the Fourier coefficients of holomorphic cusp forms} can be proven along the lines of \cite[Proposition 2.3]{Mich04}.
  For \eqref{eqn: estimate for the Fourier coefficients of Eisenstein series} we refer to \cite[Lemma 1]{BM15}.
\end{proof}

Finally, in order to handle exceptional eigenvalues, which occur in the case \( \kappa = 0 \), the following result will turn out to be useful.

\begin{lemma}\label{lemma: estimates for the exceptional eigenvalues}
  Let \( X \geq 1 \), \( n \geq 1 \) and \( \mf a \) as above.
  Assume that
  \[ X \gg X_0, \quad \text{with} \quad X_0 := \frac q{ (q, n)^\frac12 {q_0}^\frac12 \left( w, \frac qw \right) n^\frac12 }.  \]
  Then
  \[ \sum_{ t_j \text{ exc.} } \frac{ | \rho_j( \pm n, \mf a ) |^2 n }{ \cosh(\pi t_j) } X^{4i t_j} \ll (qnX)^\varepsilon \left( \frac X{X_0} \right)^{4\theta} \left( 1 + (q, n)^\frac12 {q_0}^\frac12 \left( w, \frac qw \right) \frac{ n^\frac12 }q \right), \]
  where the implicit constants only depend on \( \varepsilon \).
\end{lemma}

\begin{proof}
  We have that
  \[ \sum_{ t_j \text{ exc.} } \frac{ | \rho_j( \pm n, \mf a ) |^2 n }{ \cosh(\pi t_j) } X^{4i t_j} \ll \left( \frac X{X_0} \right)^{4\theta} \sum_{ t_j \text{ exc.} } \frac{ | \rho_j( \pm n, \mf a ) |^2 n }{ \cosh(\pi t_j) } \left( 1 + X_0 \right)^{4i t_j}. \]
  Now we use the fact that for any \( Y \geq 1 \),
  \[ \sum_{ t_j \text{ exc.} } \frac{ | \rho_j( \pm n, \mf a ) |^2 n }{ \cosh(\pi t_j) } Y^{4i t_j} \ll 1 + (qnY)^\varepsilon (q, n)^\frac12 {q_0}^\frac12 \left( w, \frac qw \right) \frac{ n^\frac12 Y }q, \]
  which can be shown the same way as in \cite[chapter 16.5]{IK04}, and the result follows.
\end{proof}

\section{Proof of Theorems \ref{thm: main theorem, smooth version} and \ref{thm: main theorem}}

Let \( w_1, w_2 : (0, \infty) \rightarrow [0, \infty) \) be smooth functions, which are compactly supported in \( [ 1/2, 1 ] \) and which satisfy
\[ \frac{ \partial^\nu w_i }{\partial \xi^\nu} (\xi) \ll \frac1{ \Omega^\nu } \quad \text{and} \quad \int \! \left| \frac{ \partial^\nu w_i }{\partial \xi^\nu} (\xi) \right| \, d\xi \ll \frac1{ \Omega^{\nu - 1} } \quad \text{for } \nu \geq 1, \]
for some \( \Omega < 1 \).
We will look at the sum
\[ D(x_1, x_2) := \sum_n w_1\left( \frac{r_1 n + f_1}{x_1} \right) w_2\left( \frac{r_2 n + f_2}{x_2} \right) d(r_1 n + f_1) d(r_2 n + f_2), \]
with the aim of showing that it can be written asymptotically as
\[ D(x_1, x_2) = M(x_1, x_2) + R(x_1, x_2), \]
where \( M(x_1, x_2) \) denotes the main term, which has the form
\begin{align} \label{eqn: main term}
  M(x_1, x_2) = \int \! w_1\left( \frac{r_1 \xi + f_1}{x_1} \right) w_2\left( \frac{r_2 \xi + f_2}{x_2} \right) P( \log( r_1 \xi + f_1 ), \log( r_2 \xi + f_2 ) ) \, d\xi
\end{align}
with a quadratic polynomial \( P(\xi_1, \xi_2) \), and where \( R(x_1, x_2) \) forms the error term.
The assumptions we hereby need to make are
\begin{align} \label{eqn: conditions on the parameters}
  f_1 \ll {x_1}^{1 - \varepsilon}, \quad f_2 \ll {x_2}^{1 - \varepsilon} \quad \text{and} \quad h \ll r_2 {x_1}^{1 - \varepsilon} \Omega^2.
\end{align}
We can also assume that
\[ {r_0}^2 {r_1}^2 r_2 \ll x_1, \]
since otherwise our results are trivial.
Furthermore note that from the first two bounds at \eqref{eqn: conditions on the parameters} and the size of the supports of \(w_1\) and \(w_2\), it follows that
\[ r_2 x_1 \asymp r_1 x_2. \]

We will prove the following three bounds for the error term:
\begin{align}
  R(x_1, x_2) &\ll r_0 (r_2 x_1)^{\frac12 + \varepsilon} \left( \frac{ |h|^\theta}{\Omega^\frac12} + (r_2 x_1)^\theta \right), \label{eqn: first main estimate} \\
  R(x_1, x_2) &\ll r_0 (r_2 x_1)^{\frac12 + \varepsilon} \left( \frac1{\Omega^\frac12} + \left( \frac{ (r_0 r_1 r_2, h) x_1 }{ r_0 {r_1}^2 r_2} \right)^\theta \left( 1 + \frac{ (r_0 r_1 r_2, h)^\frac14 |h|^\frac14 }{ {r_0}^\frac14 (r_1 r_2)^\frac12 } \right) \right), \label{eqn: second main estimate} \\
  R(x_1, x_2) &\ll r_0 (r_2 x_1)^{\frac12 + \varepsilon} \left( \frac1{\Omega^\frac12} + \left( \frac{r_2 x_1}{ |h| } \right)^\theta \left( 1 + \frac{ (r_0 r_1 r_2, h)^\frac14 |h|^\frac14 }{ {r_0}^\frac14 (r_1 r_2)^\frac12 } \right) \right). \label{eqn: third main estimate}
\end{align}
Recall that \(r_0\) was defined at \eqref{eqn: definition of r_0}.
From the first bound and the choice \( \Omega = 1 \), we immediately get Theorem \ref{thm: main theorem, smooth version}.
In order to prove Theorem \ref{thm: main theorem}, we choose
\[ \Omega = \frac{ {r_0}^\frac23 {r_1}^\frac23 {r_2}^\frac13 }{ {x_1}^\frac13 }, \]
and use the second bound for
\[ (r_0 r_1 r_2, h) h \ll \frac{ (r_1 r_2)^\frac43 }{ {r_0}^\frac13} \left( \frac{x_1}{r_1} \right)^\frac23 \left( \frac{ {r_0}^2 {r_1}^2 r_2 }{x_1} \right)^{4\theta}, \]
and the third bound for
\[ \frac{ (r_1 r_2)^\frac43 }{ {r_0}^\frac13} \left( \frac{x_1}{r_1} \right)^\frac23 \left( \frac{ {r_0}^2 {r_1}^2 r_2 }{x_1} \right)^{4\theta} \ll (r_0 r_1 r_2, h) h \ll {r_0}^\frac13 {r_1}^\frac43 {r_2}^\frac53 {x_1}^{\frac13 - \varepsilon}. \]
This way we are led to
\[ R(x_1, x_2) \ll (r_0 r_1 r_2, h)^\theta {r_0}^{\frac23 + \theta} (r_1 r_2)^\frac13 \left( \frac{x_1}{r_1} \right)^{\frac23 + \varepsilon}, \]
and Theorem \ref{thm: main theorem} follows by setting \( x_1 = r_1 x \), \( x_2 = r_2 x \) and using suitable weight functions.

Before diving into the proof, we first want to describe a smooth decomposition of the divisor function which was used by Meurman \cite{Meu01} to treat the binary additive divisor problem (and which originally goes back to Heath-Brown).
Let \( v_0 : \mathbb{R} \rightarrow [0, \infty) \) be a smooth and compactly supported function such that
\[ v_0(\xi) = 1 \quad \text{for} \quad |\xi| \leq 1, \quad \text{and} \quad v_0(\xi) = 0 \quad \text{for} \quad |\xi| \geq 2, \]
and set
\[ v(\xi) := v_0\left( \frac\xi{ \sqrt{x_2} } \right) \quad \text{and} \quad h(a, b) := v(a) ( 2 - v(b) ). \]
For \( ab \leq x_2 \), we have that
\[ ( v(a) - 1 )( v(b) - 1 ) = 0, \]
so that for \( n \leq x_2 \), it holds that
\[ d(n) = \sum_{ab = n} v(a) ( 2 - v(b) ) = \sum_{ab = n} h(a, b). \]

It will furthermore be helpful to dyadically split the supports of the variables \(a\) and \(b\).
In order to do so, we choose smooth and compactly supported functions \( h_X : (0, \infty) \rightarrow [0, \infty) \), such that
\[ \supp u_X \subset \left[ \frac X2, 2X \right], \quad \frac{ \partial^\nu u_X }{ \partial \xi^\nu }(\xi) \ll \frac1{X^\nu} \quad \text{and} \quad \sum_X u_X \equiv 1, \]
where the last sum runs over powers of \(2\).
Then we set
\[ h_{AB}(a, b) := h(a, b) u_A(a) u_B(b). \]

Back to our sum -- we split the second divisor function and use the dyadic decomposition described just before so that
\[ D(x_1, x_2) = \sum_{A,B} D_{AB}(x_1, x_2), \]
where
\begin{align*}
  D_{AB}(x_1, x_2) &:= \sum_n w_1\left( \frac{r_1 n + f_1}{x_1} \right) w_2\left( \frac{r_2 n + f_2}{x_2} \right) d(r_1 n + f_1) \sum_\twoln{a, b}{ab = r_2 n + f_2} h_{AB}(a, b) \\
    &= \sum_\twoln{a, b}{ ab \equiv f_2 \smod{r_2} } \tilde f(a, b) d\left( \frac{r_1}{r_2} (ab - f_2) + f_1 \right),
\end{align*}
and
\[ \tilde f(a, b) := w_1\left( \frac{ \frac{r_1}{r_2} ( ab - f_2 ) + f_1 }{x_1} \right) w_2\left( \frac{ab}{x_2} \right) h_{AB}(a, b). \]
Note that the variables \(A\) and \(B\), which run over powers of \(2\), satisfy
\[ AB \asymp x_2, \quad A \ll B \quad \text{and} \quad A \ll {x_2}^\frac12. \]

In the following we have to pay a lot of attention to possible common divisors between the different parameters, and it will be helpful to define for \( i = 1, 2 \),
\[ u_i := (r_i, f_i), \quad s_i := \frac{r_i}{u_i}, \quad g_i := \frac{f_i}{u_i}, \quad \text{and} \quad h := r_1 f_2 - f_1 r_2, \quad h_0 := \frac h{u_1 u_2}. \]
Now, since the product \(ab\) in the above sum must be divisible by \(u_2\), we can write
\begin{align*}
  D_{AB}(x_1, x_2) &= \sum_{u_2^\ast \mid u_2} \sum_\twoln{a}{ (a, s_2 u_2^\ast) = 1 } \sum_\twoln{b}{ ab \equiv g_2 \smod{s_2} } \tilde f\left( \frac{u_2}{u_2^\ast} a, u_2^\ast b \right)  d\left( \frac{r_1}{s_2}(ab - g_2) + f_1 \right).
\end{align*}
Choose \( \tilde a \) and \( \tilde s_2 \) such that
\[ a \tilde a + s_2 \tilde s_2 = 1, \]
so that \(b\) in the above sum has the form
\[ b = \tilde a g_2 + s_2 n \quad \text{with} \quad n \in \mathbb{Z}, \]
and hence
\begin{align*}
  D_{AB}(x_1, x_2) &= \sum_\thrln{u_2^\ast \mid u_2}{a, n}{ (a, s_2 u_2^\ast) = 1 } \tilde f\left( \frac{u_2}{u_2^\ast} a, \frac{u_2^\ast}{u_2 a} ( r_2 (an - g_2 \tilde s_2) + f_2 ) \right) d\left( r_1 (an - g_2 \tilde s_2) + f_1 \right)  \\
    &= \sum_\thrln{u_2^\ast \mid u_2}{a}{ (a, s_2 u_2^\ast) = 1 } \sum_{ n \equiv f_1 - g_2 r_1 \overline{s_2} \smod{r_1 a} } d(n) f(n; a),
\end{align*}
with
\[ f(\xi; a) := w_1\left( \frac \xi{x_1} \right) w_2\left( \frac{ \frac{r_2}{r_1}(\xi - f_1) + f_2 }{x_2} \right) h_{AB}\left( \frac{u_2}{u_2^\ast} a, \frac{u_2^\ast}{u_2 a} \left( \frac{r_2}{r_1}(\xi - f_1) + f_2 \right) \right). \]
Note that the modular inverse \( \overline{s_2} \), which occurs in the congruence condition, is understood to be mod \(a\).
Also note that the support of \( f(\xi; a) \) is given by
\[ \supp f(\bullet; a) \asymp x_1 \quad \text{and} \quad \supp f(\xi; \bullet) \asymp \frac{u_2^\ast}{u_2} A, \]
and that its derivatives can be bound by
\[ \frac{ \partial^{\nu_1 + \nu_2} f }{ \partial \xi^{\nu_1} a^{\nu_2} } (\xi; a) \ll \frac1{ (x_1 \Omega)^{\nu_1} } \left( \frac{u_2}{u_2^\ast A} \right)^{\nu_2} \quad \text{for} \quad \nu_1, \nu_2 \geq 0,\]
while also satisfying
\[ \int \! \left| \frac{ \partial^{\nu_1 + \nu_2} f }{ \partial \xi^{\nu_1} a^{\nu_2} } (\xi; a) \right| \, d\xi \ll \frac1{ (x_1 \Omega)^{\nu_1 - 1} } \left( \frac{u_2}{u_2^\ast A} \right)^{\nu_2} \quad \text{for} \quad \nu_1 \geq 1, \, \nu_2 \geq 0. \]

\subsection{Use of Voronoi summation}

We use Voronoi summation in the form of Theorem \ref{thm: Voronoi summation for d(n) in arithmetic progressions} to treat the divisor sum in arithmetic progressions.
This way we are led to
\[ D_{AB}(x_1, x_2) = \Sigma_{AB}^0 - 2\pi \Sigma_{AB}^+ + 4 \Sigma_{AB}^-, \]
with
\begin{align*}
  \Sigma_{AB}^0 &:= \frac1{r_1} \!\!\! \sum_\thrln{u_2^\ast \mid u_2}{a}{ (a, s_2 u_2^\ast) = 1 } \!\!\! \frac 1a \int \! \lambda_{ f_1 - g_2 r_1 \overline{s_2}, r_1 a }(\xi) f(\xi; a) \, d\xi, \\
  \Sigma_{AB}^\pm &:= \frac1{r_1} \!\!\! \sum_\thrln{u_2^\ast \mid u_2}{a}{ (a, s_2 u_2^\ast) = 1 } \!\!\! \sum_{c \mid r_1 a} \frac ca \sum_{n = 1}^\infty d(n) \frac{ S( f_1 - g_2 r_1 \overline{s_2}, \pm n; c ) }{c^2} \int \! B^\pm\left( \frac{4\pi}c \sqrt{n \xi} \right) f(\xi; a) \, d\xi,
\end{align*}
and
\[ B^+(\xi) := Y_0(\xi) \quad \text{and} \quad B^-(\xi) := K_0(\xi). \]
The main term will be extracted from \( \Sigma_{AB}^0 \), but we will postpone this until the end and take care first of \( \Sigma_{AB}^\pm \).

We reshape these sums a little bit,
\begin{align*}
  \Sigma_{AB}^\pm &= \frac1{r_1} \sum_{u_2^\ast \mid u_2} \sum_\thrln{a, c}{c \mid r_1 a}{ (a, s_2 u_2^\ast) = 1 } (\ldots) = \frac1{r_1} \sum_\twoln{u_2^\ast \mid u_2}{r_1^\ast \mid r_1} \sum_\twoln{d}{ (d, r_1) = r_1^\ast } \sum_\thrln{a, c}{ dc = r_1 a }{ (a, s_2 u_2^\ast) = 1 } (\ldots) \\
    &= \frac1{r_1} \sum_\twoln{u_2^\ast \mid u_2}{r_1^\ast \mid r_1} \sum_\twoln{d}{ (d, r_1^\ast s_2 u_2^\ast) = 1 } \sum_\twoln{c}{ \left( c, s_2 u_2^\ast \right) = 1 } (\ldots),
\end{align*}
where we have to replace \(c\) by \( r_1^\ast c \) and \(a\) by \(dc\), so that
\begin{align*}
  \Sigma_{AB}^\pm &= \sum_\twoln{u_2^\ast \mid u_2}{r_1^\ast \mid r_1} \sum_\twoln{d}{ (d, r_1^\ast s_2 u_2^\ast) = 1 } \frac{ R_{AB}^\pm }d,
\end{align*}
with
\[ R_{AB}^\pm := \sum_\twoln{c}{ \left( c, s_2 u_2^\ast \right) = 1 } \sum_{n = 1}^\infty d(n) \frac{ S(f_1 - g_2 r_1 \overline{s_2}, \pm n; r_1^\ast c) }{ r_1^\ast c } F^\pm(r_1^\ast c; dc, n), \]
and
\[ F^\pm(\eta; a, n) := \frac{r_1^\ast}{\eta r_1} \int \! B^\pm\left( \frac{4\pi}\eta \sqrt{n \xi} \right) f(\xi; a) \, d\xi. \]
As a reminder, the modular inverse \( \overline{s_2} \) occuring in the Kloosterman sum is now understood to be mod \(dc\).

Let
\[ N_0^- := \frac{ {x_1}^\varepsilon }{x_1} {A^\ast}^2, \quad N_0^+ := \frac{ {x_1}^\varepsilon }{x_1 \Omega^2} {A^\ast}^2 \quad \text{and} \quad A^\ast := \frac{u_2^\ast}{u_2} \frac{r_1^\ast A}d. \]
Regarding \( F^\pm(r_1^\ast c; dc, n) \), we have the bounds
\begin{align*}
  F^+(r_1^\ast d; dc, n) &\ll \frac{ (x_1 \Omega)^\frac12 }{n^\frac12} \left( \frac{A^\ast}{ \sqrt{n x_1} \Omega } \right)^{\nu - \frac12}, \\
  F^-(r_1^\ast d; dc, n) &\ll \frac{ {x_1}^\frac12 }{n^\frac12} \left( \frac{A^\ast}{ \sqrt{n x_1} } \right)^{\nu - \frac12},
\end{align*}
which can be shown using \eqref{eqn: consequence of the recurrence relations} resp. \eqref{eqn: bounds for the K-Bessel function}.
With the help of these bounds, it is not hard to see that the sum over \(n\) in \( R_{AB}^\pm \) can be cut at \( N_0^\pm \).
After dyadically dividing the remaining sum, we are left with
\[ R_{AB}^\pm(N) := \sum_\twoln{c}{ \left( c, s_2 u_2^\ast \right) = 1 } \sum_{N < n \leq 2N} d(n) \frac{ S(f_1 - g_2 r_1 \overline{s_2}, \pm n; r_1^\ast c) }{ r_1^\ast c } F^\pm(r_1^\ast c; dc, n). \]

\subsection{Treatment of the Kloosterman sums}

Not surprisingly we would like to treat the sum of Kloosterman sums occuring in \( R_{AB}^\pm(N) \) with the Kuznetsov trace formula.
However, in our situation this does not seem to be possible directly.
To deal with this difficulty, we factor out the part of the variable \( r_1^\ast \) which has the same prime factors as \(s_2 u_2^\ast\),
\[ v := (r_1^\ast, (s_2 u_2^\ast)^\infty), \quad t_1 := \frac{ r_1^\ast }v, \]
and use the twisted multiplicativity of Kloosterman sums,
\[ \frac { S(f_1 - g_2 r_1 \overline{s_2}, \pm n; r_1^\ast c) }{ r_1^\ast c } = \frac{ S\left( f_1 \overline{c t_1}, \pm n \overline{c t_1}; v \right) }v \frac { S\left( h_0 u_1, \pm n \overline{v}^2 \overline{s_2}; c t_1 \right) }{ c t_1 }. \] 
Here, all the modular inverses are finally understood to be modulo the respective modulus of the Kloosterman sum.
Obviously the first factor still depends on \(c\), but here we follow an idea of Blomer and Mili\'{c}evi\'{c} \cite{BM15} and use Dirichlet characters to separate this variable.
We define
\[ \hat S_v(\chi; n) := \sum_\twoln{y \smod v}{ (y, v) = 1 } \overline\chi(y) \frac{ S(f_1 \overline y, \pm n \overline y; v) }v, \]
where \(\chi\) is a Dirichlet character modulo \(v\), so that by the orthogonality relations of Dirichlet characters we have that
\[ \frac{ S\left( f_1 \overline{c t_1}, \pm n \overline{c t_1}; v \right) }v = \frac1{ \varphi(v) } \sum_{\chi \bmod v} \overline\chi(c t_1) \hat S_v(\overline\chi; n), \]
where the sum runs over all Dirichlet characters modulo \(v\).
Hence
\[ R_{AB}^\pm(N) = \frac1{ \varphi(v) } \sum_{\chi \bmod v} \overline\chi(t_1) R_{AB}^\pm(N; \chi), \]
with
\begin{align*}
  R_{AB}^\pm(N; \chi) &:=  \sum_{N < n \leq 2N} d(n) \hat S_v(\overline\chi; n) K_{AB}^\pm(\chi; n), \\
\intertext{and}
  K_{AB}^\pm(n; \chi) &:= \sum_\twoln{c}{ (c, s_2 u_2^\ast) = 1 } \frac{ S\left( h_0 u_1 u_2^\ast, \pm n \overline{s_2 u_2^\ast v^2}; t_1 c \right) }{t_1 c} \overline\chi(c) F^\pm( r_1^\ast c; dc, n ).
\end{align*}

Of course it is important to have good bounds for \( \hat S_v(\chi; n) \).
Directly using Weil's bound for Kloosterman sums we get
\[ \vcent{ \hat S_v(\chi; n) } \ll (f_1, n, v)^\frac12 v^{\frac12 + \varepsilon}, \]
however this can be improved with a little bit of effort.
More precisely, we will prove
\begin{align} \label{eqn: bound for hat S}
  \hat S_v(\chi; n) \ll \left( f_1, n, \frac v{ \operatorname{cond}(\chi) } \right) v^\varepsilon,
\end{align}
where \( \operatorname{cond}(\chi) \) is the conducter of \(\chi\).
The sum actually vanishes in a lot of cases, in particular when \(f_1\), \(n\) and \(v\) have certain common factors, but this result will be sufficient for our purposes.
At this point we also want to mention that
\begin{align} \label{eqn: sum over conductors}
  \frac1{ \varphi(v) } \sum_{\chi \bmod v} \frac v{ \operatorname{cond}(\chi) } &= \frac v{ \varphi(v) } \sum_{v^\ast \mid v} \frac1{v^\ast} \sum_\twoln{\chi \bmod v}{ \operatorname{cond}(\chi) = v^\ast } 1 \leq \frac v{ \varphi(v) } d(v) \ll v^\varepsilon,
\end{align}
which will be important later.

In order to prove \eqref{eqn: bound for hat S}, note first that \( \hat S_v(\chi; n) \) is quasi-multiplicative in the sense that if \( v = v_1 v_2 \) with coprime \(v_1\) and \(v_2\), and \( \chi = \chi_1 \chi_2 \) with the corresponding Dirichlet characters \( \chi_1 \) (mod \(v_1\)) and \( \chi_2 \) (mod \(v_2\)), then
\[ \hat S_v(\chi; n) = \chi_1(v_2) \chi_2(v_1) \hat S_{v_1}( \chi_1; n ) \hat S_{v_2}( \chi_2; n ). \]
It is therefore enough to look at the case where \(v\) is a prime power \( v = p^k \).

Assume first that \( \chi = \chi_0 \) is the principal character.
For \( v = p \) we have
\[ \hat S_p(\chi; n) = \frac1p \sum_\twoln{ x, y \smod p}{ (x, p) = 1 } \e\left( \frac{ y( f_1 x \pm n \overline x ) }p \right) - \frac{ \varphi(p) }p = \sum_\twoln{x \smod p}{ f_1 x \pm n \overline x \equiv 0 \smod p } 1 - \frac{ \varphi(p) }p \ll (f_1, n, p), \]
and for prime powers \( v = p^k \), \( k \geq 2 \), we have
\begin{align*}
  \hat S_{p^k}(\chi; n) &= \frac1{p^k} \sum_\twoln{ x, y \smod{p^k} }{ (x, p) = 1 } \e\left( \frac{ y(f_1 x \pm n \overline x) }{p^k} \right) - \frac1{p^k} \sum_\twoln{ x \smod{p^k}, y \smod{ p^{k - 1} } }{ (x, p) = 1 } \e\left( \frac{ y(f_1 x \pm n \overline x) }{ p^{k - 1} } \right) \\
    &= \#\left\{ x \smod{p^k} \middle| f_1 x \pm n \overline x \equiv 0 \smod{p^k} \right\} - \frac1p \#\left\{ x \smod{p^k} \middle| f_1 x \pm n \overline x \equiv 0 \smod{ p^{k - 1} } \right\} \\
    &\ll \left( f_1, n, p^k \right).
\end{align*}

In the following we can now assume that \(\chi\) is non-principal.
For \( v = p \) prime this means that \(\chi\) is primitive and hence
\begin{align*}
  \hat S_p(\chi; n) &= \frac1p \sum_\twoln{x, y \smod p}{ (xy, p) = 1 } \chi(y) \e\left( \frac{ y( f_1 x \pm n \overline x ) }p \right) \\
    &= \frac1p \sum_\twoln{x, y \smod p}{ f_1 x \pm n \overline x \not\equiv 0 \smod p } \chi(y) \overline\chi( f_1 x \pm n \overline x ) \e\left( \frac yp \right) - \frac1p \sum_\twoln{x, y \smod p}{ f_1 x \pm n \overline x \equiv 0 \smod p } \chi(y) \\
    &= \frac{ \tau(\chi) }p \left( \vcent{ \sum_\twoln{x \smod p}{ (x, p) = 1 } \overline\chi( f_1 x \pm n \overline x ) } \right) \\
    &\ll 1,
\end{align*}
where we have used the fact that both the Gau{\ss} sum \( \tau(\chi) \) and the character sum on the right are bounded by \( \BigO{\sqrt p} \), which is well-known for the former and follows from Weil's work for the latter (see e.g. \cite[Theorem 11.23]{IK04} or \cite[Chapter 6, Theorem 3]{Li96}).

It remains to look at the case of \(\chi\) having modulus \( v = p^k \), \( k \geq 2 \), which is slightly more complicated.
Let \(\chi\) be induced by the primitive character \( \chi^\ast \) of modulus \( v^\ast = p^{ k^\ast } \), and set \( v^\circ := p^{k - k^\ast} \).
In our sum
\[ \hat S_{p^k}(\chi; n) = \frac1{p^k} \sum_\twoln{ x \smod{p^k} }{ \left( x, p^k \right) = 1 } \sum_{ y \smod{p^k} } \chi(y) \e\left( \frac{ y(f_1 x \pm n \overline x) }{p^k} \right) \]
we parametrize \(y\) by
\[ y = y_1 + v^\ast y_2, \quad \text{with} \quad y_1 \bmod v^\ast \quad \text{and} \quad y_2 \bmod v^\circ. \]
Then
\begin{align*}
  \hat S_{p^k}(\chi; n) &= \frac1v \sum_\twoln{x \smod v}{ (x, v) = 1 } \sum_{ y_1 \smod{v^\ast} } \chi^\ast(y_1) \e\left( \frac{ y_1(f_1 x \pm n \overline x) }v \right) \sum_{ y_2 \smod{ v^\circ } } \e\left( \frac{ y_2 (f_1 x \pm n \overline x) }{v^\circ} \right) \\
    &= \frac1{v^\ast} \sum_\thrln{x \smod v}{ (x, v) = 1 }{ f_1 x \pm n \overline x \equiv 0 \smod{ v^\circ } } \sum_{ y_1 \smod{v^\ast} } \chi^\ast(y_1) \e\left( \frac{ y_1(f_1 x \pm n \overline x) }v \right) \\
    &= \frac{ \tau(\chi^\ast) }{v^\ast} \sum_\thrln{x \smod v}{ (x, v) = 1 }{ f_1 x \pm n \overline x \equiv 0 \smod{ v^\circ } } \overline{\chi^\ast}\left( \frac{ f_1 x \pm n \overline x }{v^\circ} \right).
\end{align*}
We set
\[ \tilde v^\circ := \frac{v^\circ}{ (f_1, n, v^\circ) }, \quad \tilde v := v^\ast \tilde v^\circ, \quad \tilde f_1 := \frac{f_1}{ (f_1, n, v^\circ) } \quad \text{and} \quad \tilde n := \frac n{ (f_1, n, v^\circ) }, \]
and the sum becomes
\[ \hat S_{p^k}(\chi; n) = (f_1, n, v^\circ) \frac{ \tau(\chi^\ast) }{v^\ast} \sum_\twoln{ x \smod{\tilde v} }{ \tilde f_1 x \pm \tilde n \overline x \equiv 0 \smod{ \tilde v^\circ } } \overline{\chi^\ast}\left( \frac{ \tilde f_1 x \pm \tilde n \overline x }{\tilde v^\circ} \right). \]
If \( \tilde v^\circ = 1 \), we have square-root cancellation for the character sum on the right (see \cite[Theorem 2]{YZH03}), so that \( \hat S_{p^k}(\chi; n) \ll (f_1, n, v^\circ) \).

Otherwise note that both \( \tilde f_1 \) and \( \tilde n \) have to be coprime with \(p\), as otherwise the sum is empty.
We parametrize \(x\) by
\[ x = x_1 \left( 1 + \tilde v^\circ x_2 \right), \quad \text{with} \quad x_1 \bmod \tilde v^\circ, \quad \left( x_1, \tilde v^\circ \right) = 1 \quad \text{and} \quad x_2 \bmod v^\ast. \]
In this case we can write \( \overline x \bmod \tilde v \) in the following way
\[ \overline x \equiv \overline{x_1} \left( 1 - \tilde v^\circ x_2 \overline{ (1 + \tilde v^\circ x_2) } \right) \bmod \tilde v, \]
and after putting this in our sum, we have
\[ \hat S_{p^k}(\chi; n) = (f_1, n, v^\circ) \frac{ \tau(\chi^\ast) }{v^\ast} \sum_\twoln{x_1 \smod{ \tilde v^\circ} }{ \tilde f_1 x_1 \pm \tilde n \overline{x_1} \equiv 0 \smod{ \tilde v^\circ } } \sum_{ x_2 \smod{v^\ast} } \overline{\chi^\ast}( P(x_2) ), \]
where \( P(X) \) is the rational function
\[ P(X) := \frac{ \tilde f_1 x_1 \tilde v^\circ X^2 + 2 \tilde f_1 x_1 X + \frac{ \tilde f_1 x_1 \pm \tilde n \overline{x_1} }{\tilde v^\circ} }{ \tilde v^\circ X + 1 }. \]

If \( p \geq 3 \), we can use \cite[Theorem 1.1]{CLZh03} to get that
\[ \sum_{ x_2 \smod{v^\ast} } \overline{\chi^\ast}( P(x_2) ) \ll 1. \]
If \( p = 2 \) and \( \tilde v^\circ \geq 8 \), we rewrite this sum
\[ \sum_{ x_2 \smod{v^\ast} } \overline{\chi^\ast}( P(x_2) ) = \sum_{ x_2 \smod{2 v^\ast} } \overline{\chi^\ast}\left( P\left( \frac{x_2}2 \right) \right) = 2 \sum_{ x_2 \smod{v^\ast} } \overline{\chi^\ast}\left( P\left( \frac{x_2}2 \right) \right), \]
so that we can again apply the cited theorem to show that this sum is \( \BigO{1} \).
Finally for the remaining cases \( \tilde v^\circ = 2 \) and \( \tilde v^\circ = 4 \), we can use \cite[Theorem 2.1]{CLZh03} to show square-root cancellation.
This concludes the proof of \eqref{eqn: bound for hat S}.

\subsection{Auxiliary estimates}

We want to use the Kuznetsov trace formula in the form \eqref{eqn: Kuznetsov formula for q = rs} with
\[ \tilde q := t_1 s_2 u_2^\ast v^2, \quad \tilde r := s_2 u_2^\ast v^2, \quad \tilde s := t_1 \quad \tilde q_0 := v, \quad \tilde m := h_0 u_1 u_2^\ast, \quad \tilde n := n. \]
However, before we can do so some technical arrangements have to be made.
Set
\[ \tilde F^\pm (c; n) := h(n) \frac{r_1^\ast}{r_1} \frac{ v \sqrt{s_2 u_2^\ast} }{4\pi} \sqrt{ \frac{r_2}{n|h|} } \int \! c B^\pm\left( c \sqrt{ \xi \frac{r_2}{ |h| } } \right) f\left( \xi; 4\pi \frac{ d \sqrt{n} }{r_1^\ast c} \sqrt{ \frac{ |h| }{r_2} } \right) \, d\xi, \]
where \(h\) is a smooth and compactly supported bump function such that
\[ h(n) = 1 \quad \text{for} \quad n \in [N, 2N], \quad \supp h \asymp N \quad \text{and} \quad h^{ (\nu) }(n) \ll \frac1{N^\nu}. \]
We have defined this function in such a way that
\[ F^\pm( r_1^\ast c; dc, n ) = \frac1{ \sqrt{\tilde r} } \tilde F^\pm\left( \frac{ 4\pi \sqrt{ |\tilde m \tilde n| } }{ \sqrt{\tilde r} \tilde s c }; n \right) \quad \text{for} \quad n \in [N, 2N]. \]
Note that
\[ \supp \tilde F^\pm(\bullet; n) \asymp C := \frac1{A^\ast} \sqrt{ \frac{ N |h| }{r_2} }, \quad \tilde F^\pm(c; n) \ll v \sqrt{s_2 u_2^\ast} \frac{ r_1^\ast }{A^\ast r_1} {x_1}^{1 + \varepsilon}. \]
We need to seperate the variable \(n\) to be able to use the large sieve inequalities later, and to this end we make use of Fourier inversion,
\[ \tilde F^\pm(c; n) = \int \! G_0(\lambda) G_\lambda^\pm(c) e(\lambda n) \, d\lambda, \quad G_\lambda^\pm(c) := \frac1{ G_0(\lambda) } \int \! \tilde F^\pm(c; n) e(-\lambda n) \, dn, \]
where
\[ G_0(\lambda) := v \sqrt{s_2 u_2^\ast} \frac{r_1^\ast}{A^\ast r_1} {x_1}^{1 + \varepsilon} \min\left( N, \frac1{N \lambda^2} \right). \]
Eventually, our sum of Kloosterman sums looks like
\[ K_{AB}^\pm(\chi; n) := \int \! G_0(\lambda) e(\lambda n) \sum_\twoln{c}{ (c, \tilde r) = 1 } \overline\chi(c) \frac{ S\left( \tilde m, \pm \tilde n \overline{\tilde r}; \tilde s c \right) }{ c \tilde s \sqrt{\tilde r} } \tilde G_\lambda^\pm\left( 4\pi \frac{ \sqrt{ |\tilde m \tilde n| } }{ c \tilde s \sqrt{\tilde r} } \right) \, d\lambda. \]

Next, we need to find good estimates for the Bessel transforms occuring in the Kuznetsov formula.
For convenience set
\[ C := \frac1{A^\ast} \sqrt{ \frac{ |h| N }{r_2} } \quad \text{and} \quad Z := \frac1{A^\ast} \sqrt{x_1 N}. \]
Note that due to the assumptions made at \eqref{eqn: conditions on the parameters}, it is true that \( C \ll 1 \).

\begin{lemma} \label{lemma: bounds for the Bessel transforms of G}
  If \( N \ll N_0^- \), we have
    \begin{align}
      \tilde G_\lambda^\pm(it), \check G_\lambda^\pm(it) &\ll C^{-2t} &\text{for} \quad 0 \leq t < \frac14, \label{eqn: First bound for the Bessel transforms of G} \\
      \frac{ \tilde G_\lambda^\pm(t) }{ (1 + t)^\kappa }, \check G_\lambda^\pm(t), \dot G_\lambda^\pm(t) &\ll \frac{ {x_1}^\varepsilon }{1 + t^\frac52} &\text{for} \quad t \geq 0. \label{eqn: Second bound for the Bessel transforms of G}
    \end{align}
  If \( N_0^- \ll N \ll N_0^+ \), we have for any \( \nu \geq 0 \),
    \begin{align}
      \tilde G_\lambda^\pm(it), \check G_\lambda^\pm(it) &\ll {x_1}^{-\nu} &\text{for} \quad 0 \leq t < \frac14, \label{eqn: Third bound for the Bessel transforms of G} \\
      \frac{ \tilde G_\lambda^\pm(t) }{ (1 + t)^\kappa }, \check G_\lambda^\pm(t), \dot G_\lambda^\pm(t) &\ll \frac{ {x_1}^\varepsilon }{ Z^\frac52 } \left( \frac Zt \right)^\nu &\text{for} \quad t \geq 0. \label{eqn: Fourth bound for the Bessel transforms of G}
    \end{align}
\end{lemma}

\begin{proof}
  Since all occuring integrals can be interchanged, we can look directly at the Bessel transforms inside \( \tilde F^\pm(c, n) \) and their first two partial derivatives in \(n\).
  We will confine ourselves with the treatment of \( \tilde F^\pm(c, n) \), since the corresponding estimates for the derivatives can be shown the same way.
  
  First we want to prove the first two bounds, which hold when \( N \ll N_0^- \).
  Here again, we can look directly at the function inside the integral over \(\xi\), given by
  \[ H_1(c) := c B^\pm\left( c \sqrt{ \frac{\xi r_2}{ |h| } } \right) f\left( \xi; 4\pi \frac d{r_1^\ast c} \sqrt{ \frac{ n |h| }{r_2} } \right). \]
  We have that
  \[ \supp H_1 \asymp C \quad \text{and} \quad H_1^{ (\nu) }(c) \ll {x_1}^\varepsilon C \left( \frac{ {x_1}^\varepsilon }C \right)^\nu, \]
  so that by Lemma \ref{lemma: estimates for the Bessel transforms},
  \begin{align*}
    \tilde H_\lambda^\pm(it), \check H_\lambda^\pm(it) &\ll C^{1 - 2t} &\text{for} \quad 0 \leq t < \frac14, \\
    \frac{ \tilde H_\lambda^\pm(t) }{ (1 + t)^\kappa }, \check H_\lambda^\pm(t), \dot H_\lambda^\pm(t) &\ll \frac{ {x_1}^\varepsilon C  }{1 + t^\frac52} &\text{for} \quad t \geq 0,
  \end{align*}
  from which we get \eqref{eqn: First bound for the Bessel transforms of G} and \eqref{eqn: Second bound for the Bessel transforms of G}.
  
  Now assume \( N_0^- \ll N \ll N_0^+ \).
  By using Lemma \ref{lemma: Bessel functions as oscillating functions} and partially integrating once over \(\xi\), we get
  \[ \tilde F^+(c) = \frac1\pi \frac{ h(n) }{ \sqrt{n} } \frac{ r_1^\ast v \sqrt{s_2 u_2^\ast} }{r_1} \Im \left( \int \! \e\left( \frac c{2\pi} \sqrt{ \frac{\xi r_2}{ |h| } } \right) \tilde w(c) \, d\xi \right) \]
  with
  \[ \tilde w(c) := \frac\partial{\partial \xi} \left( \sqrt\xi v_Y\left( \frac c\pi \sqrt{ \frac{\xi r_2}{ |h| } } \right) f\left( \xi; 4\pi \frac d{r_2^\ast c} \sqrt{ \frac{ n |h| }{r_2} } \right) \right). \]
  It is hence enough to look at
  \[ H_2(c) := \e\left( \frac c{2\pi} \sqrt{ \frac{\xi r_2}{ |h| } } \right) \tilde w(c). \]
  Note that
  \[ \supp \tilde w \asymp C, \quad \text{and} \quad \tilde w^{ (\nu) } \ll \frac{ \omega(\xi) }{ {x_1}^\frac12 Z^\frac12 } \frac1{C^\nu}, \]
  where
  \[ \omega(\xi) := 1 + \left| w_1'\left( \frac\xi{x_1} \right) \right| + \left| w_2'\left( \frac{ \frac{r_2}{r_1} (\xi - f_1) + f_2 }{x_2} \right) \right|. \]
  Here we use Lemma \ref{lemma: estimates for the Bessel transforms of oscillating functions for large alpha} with \( \alpha = \sqrt{ \frac{\xi r_2}{ |h| } } \) and \( X = C \).
  This is possible as
  \[ \alpha X \gg \frac{ ( x_1 N_0^-)^\frac12 }{ A^\ast } \gg {x_1}^\varepsilon, \]
  and we get
  \begin{align*}
    \tilde H_\lambda^\pm(it), \check H_\lambda^\pm(it) &\ll {x_1}^{-\nu} &\text{for} \quad 0 \leq t < \frac14, \\
    \frac{ \tilde H_\lambda^\pm(t) }{ (1 + t)^\kappa }, \check H_\lambda^\pm(t), \dot H_\lambda^\pm(t) &\ll \frac{ {x_1}^\varepsilon }{ {x_1}^\frac12 Z^\frac32 } \left( \frac Zt \right)^\nu &\text{for} \quad t \geq 0,
  \end{align*}
  and \eqref{eqn: Third bound for the Bessel transforms of G} and \eqref{eqn: Fourth bound for the Bessel transforms of G} follow immediately.
\end{proof}

\subsection{Use of the Kuznetsov trace formula}

Here we will only look at \( K_{AB}^+(\chi; n) \) and we will assume that \( h > 0 \), since all other cases can be treated in essentially the same way.

A use of Theorem \ref{thm: Kuznetsov trace formula} gives
\[ R_{AB}^+(N; \chi) = \int \! G_0(\lambda) \left( \Xi_1(N) + \Xi_2(N) + \Xi_3(N) \right) \, d\lambda, \]
where
\begin{align*}
  \Xi_1(N) &:= \sum_{j = 1}^\infty \frac{ \tilde G_\lambda^+(t_j) }{ ( 1 + |t_j| )^\kappa } \left( \frac{ ( 1 + |t_j| )^\frac\kappa2 }{ \sqrt{ \cosh(\pi t_j) } } \overline{ \rho_j }(\tilde m, \infty) \sqrt{\tilde m} \right) \Sigma_j^{ (1) }(N), \\
  \Xi_2(N) &:= \sum_{ \mathfrak{c} \text{ sing.} } \frac1{4\pi} \int_{-\infty}^\infty \! \frac{ \tilde G_\lambda^+(r) }{ ( 1 + |t| )^\kappa } \left( \frac{ ( 1 + |t| )^\frac\kappa2 }{ \sqrt{ \cosh(\pi r) } } \overline{ \varphi_{ \mathfrak{c}, r } } \left( \tilde m, \infty \right) \sqrt{\tilde m} \right) \Sigma_{ \mathfrak{c}, r }^{ (2) }(N) \, dr, \\
  \Xi_3(N) &:= \sum_\twoln{ k \equiv \kappa \smod 2, \,\, k > \kappa }{ 1 \leq j \leq \theta_k(q, \chi) } \dot G_\lambda^+(k) \left( \sqrt{ (k - 1)! } \, \overline{ \psi_{j, k} }(\tilde m, \infty) \sqrt{\tilde m} \right) \Sigma_{j, k}^{ (3) }(N),
\end{align*}
with
\begin{align*}
  \Sigma_j^{ (1) }(N) &:= \frac{ ( 1 + |t_j| )^\frac\kappa2 }{ \sqrt{ \cosh(\pi t_j) } } \sum_{N < n \leq 2N} d(n) \hat S_v(\overline\chi; n) \e\left( \lambda n - n \frac{ \overline{\tilde s} }{\tilde r} \right) \rho_j\left( n, \frac1{\tilde s} \right) \sqrt{n}, \\
  \Sigma_{ \mathfrak{c}, r }^{ (2) }(N) &:= \frac{ ( 1 + |t| )^\frac\kappa2 }{ \sqrt{ \cosh(\pi t) } } \sum_{N < n \leq 2N} d(n) \hat S_v(\overline\chi; n) \e\left( \lambda n - n \frac{ \overline{\tilde s} }{\tilde r} \right) \varphi_{ \mathfrak{c}, r }\left( n, \frac1{\tilde s} \right) \sqrt{n}, \\
  \Sigma_{j, k}^{ (3) }(N) &:= \sqrt{ (k - 1)! } \sum_{N < n \leq 2N} d(n) \hat S_v(\overline\chi; n) \e\left( \lambda n - n \frac{ \overline{\tilde s} }{\tilde r} \right) \psi_{j, k}\left( n, \frac1{\tilde s} \right) \sqrt{n}.
\end{align*}

Assume first that \( N \ll N_0^- \).
We divide \( \Xi_1(N) \) into three parts:
\[ \Xi_1(N) = \sum_{t_j \leq 1} (\ldots) + \sum_{t_j > 1} (\ldots) + \sum_{ t_j \text{ exc.} } (\ldots)  =: \Xi_{ 1 \text{a} }(N) + \Xi_{ 1 \text{b} }(N) + \Xi_{ 1 \text{c} }(N). \]
We use Cauchy-Schwarz on \( \Xi_{ 1 \text{a} }(N) \), and then Lemma \ref{lemma: bounds for the Bessel transforms of G}, Theorem \ref{thm: large sieve inequalities} and Lemma \ref{lemma: estimates for Fourier coefficients, II} to bound the different factors, which leads to
\begin{align*}
  \Xi_{ 1 \text{a} }(N) &\leq \max_{t_j \leq 1} \left| \frac{ \tilde G_\lambda^+(t_j) }{ (1 + t_j)^\kappa } \right| \left( \vcent{ \sum_{t_j \leq 1} \frac{ (1 + t_j)^\kappa  }{ \cosh(\pi t_j) } | \rho_j(\tilde m, \infty) |^2 \tilde m } \right)^\frac12 \left( \vcent{ \sum_{t_j \leq 1} \left| \Sigma_j^{ (1) }(N) \right|^2 } \right)^\frac12 \\
    &\ll {x_1}^\varepsilon \tilde m^\theta \left( 1 + {\tilde q_0}^\frac12 \frac N{\tilde q} \right)^\frac12 \left( N^\varepsilon v^\varepsilon \sum_{N < n \leq 2N} ( f_1, n, v^\circ )^2 \right)^\frac12 \\
    &\ll {v^\circ}^\frac12 {x_1}^\varepsilon h^\theta \frac{A^\ast}{x_1} \left( {x_1}^\frac12 + \frac{A^\ast}{ v^\frac14 (r_1^\ast s_2 u_2^\ast)^\frac12 } \right),
\end{align*}
where we have set
\[ v^\circ := \frac v{ \operatorname{cond}(\overline\chi) }. \]
We split up \( \Xi_{ 1 \text{b} }(N) \) into dyadic segments
\[ \Xi_{ 1 \text{b} }(N, T) := \sum_{T < t_j \leq 2T} \tilde G_\lambda^+(t_j) \frac{ \overline{ \rho_j }(\tilde m, \infty) \sqrt{\tilde m} }{ \sqrt{ \cosh(\pi t_j) } } \Sigma_j^{ (1) }(N), \]
and in the same way as above we can show that
\[ \Xi_{ 1 \text{b} }(N, T) \ll {v^\circ}^\frac12 {x_1}^\varepsilon h^\theta \frac{A^\ast}{T^\frac12 x_1} \left( {x_1}^\frac12 + \frac{A^\ast}{ v^\frac14 (r_1^\ast s_2 u_2^\ast)^\frac12 } \right), \]
which gives the same bound for \( \Xi_{ 1 \text{b} }(N) \) as for \( \Xi_{ 1 \text{a} }(N) \).
Finally, for \( \Xi_{ 1 \text{c} } \) we get
\[ \Xi_{ 1 \text{c} }(N) \ll {v^\circ}^\frac12 {x_1}^\varepsilon (r_2 x_1)^\theta \frac{A^\ast}{x_1} \left( {x_1}^\frac12 + \frac{ A^\ast }{ v^\frac14 (r_1^\ast s_2 u_2^\ast)^\frac12 } \right), \]
and all in all this leads to
\begin{align}
  \int \! G_0(\lambda) \Xi_1(N) \, d\lambda \ll {v^\circ}^\frac12 v (r_2 x_1)^{\frac12 + \theta + \varepsilon}. \label{eqn: first preliminary estimate}
\end{align}

In exactly the same manner, but using Lemma \ref{lemma: estimates for Fourier coefficients, I} instead of Lemma \ref{lemma: estimates for Fourier coefficients, II}, we can also get the bounds
\begin{align*}
  \Xi_{ 1 \text{a} }(N), \, \Xi_{ 1 \text{b} }(N) &\ll {v^\circ}^\frac12 {x_1}^\varepsilon \frac{A^\ast}{x_1} \left( {x_1}^\frac12 + \frac{A^\ast}{ v^\frac14 (r_1^\ast s_2 u_2^\ast)^\frac12 } \right) \left( 1 + \frac{ (r_1^\ast r_2 v , h )^\frac14 h^\frac14 }{ (r_1^\ast r_2)^\frac12 v^\frac14} \right),
\end{align*}
and
\begin{align*}
  \Xi_{ 1 \text{c} }(N) &\ll {v^\circ}^\frac12 {x_1}^\varepsilon \left( \frac{r_2 x_1}h \right)^\theta \frac{A^\ast}{x_1} \left( {x_1}^\frac12 + \frac{A^\ast}{ v^\frac14 (r_1^\ast s_2 u_2^\ast)^\frac12 } \right) \left( 1 + \frac{ (r_1^\ast r_2 v , h )^\frac14 h^\frac14 }{ (r_1^\ast r_2)^\frac12 v^\frac14} \right),
\end{align*}
so that
\begin{align}
  \int \! G_0(\lambda) \Xi_1(N) \, d\lambda \ll {v^\circ}^\frac12 v (r_2 x_1)^{\frac12 + \varepsilon} \left( \frac{r_2 x_1}{h} \right)^\theta \left( 1 + \frac{ (r_1 r_2 v , h )^\frac14 h^\frac14 }{ (r_1 r_2)^\frac12 v^\frac14} \right). \label{eqn: second preliminary estimate}
\end{align}
Furthermore since
\[ \frac{ {x_1}^\varepsilon }C \gg \frac{ {x_1}^\frac12 {r_2}^\frac12 }{ h^\frac12 } \gg \frac{ r_1^\ast r_2 v^\frac12 }{ (r_1^\ast r_2 v, h)^\frac12 h^\frac12 } = \frac{\tilde q}{ (\tilde q, \tilde m)^\frac12 {\tilde q_0}^\frac12 \tilde m^\frac12 }, \]
we can also make use of Lemma \ref{lemma: estimates for the exceptional eigenvalues} here, so that
\begin{align*}
  \Xi_{ 1 \text{c} }(N) &\ll \left( \vcent{ \sum_{t_j \leq 1} \frac{ | \rho_j(\tilde m, \infty) |^2 \tilde m }{ \cosh(\pi t_j) } \left( \frac{ {x_1}^\varepsilon  }C \right)^{4i t_j} } \right)^\frac12 \left( \vcent{ \sum_{t_j \leq 1} \left| \Sigma_j^{ (1) }(N) \right|^2 } \right)^\frac12 \\
    &\ll {v^\circ}^\frac12 {x_1}^{\theta + \varepsilon} \frac{ (r_1^\ast r_2 v, h)^\theta }{ ({r_1^\ast}^2 r_2 v)^\theta } \frac{A^\ast}{x_1} \left( {x_1}^\frac12 + \frac{A^\ast}{ v^\frac14 (r_1^\ast s_2 u_2^\ast)^\frac12 } \right) \left( 1 + \frac{ (r_1^\ast r_2 v , h )^\frac14 h^\frac14 }{ (r_1^\ast r_2)^\frac12 v^\frac14} \right),
\end{align*}
and hence
\begin{align}
  \int \! G_0(\lambda) \Xi_1(N) \, d\lambda \ll {v^\circ}^\frac12 v (r_2 x_1)^{\frac12 + \varepsilon} \left( x_1 \frac{ (r_1 r_2 v, h) }{ {r_1}^2 r_2 v } \right)^\theta \left( 1 + \frac{ (r_1 r_2 v , h )^\frac14 h^\frac14 }{ (r_1 r_2)^\frac12 v^\frac14} \right). \label{eqn: third preliminary estimate}
\end{align}

Now assume \( N_0^- \ll N \ll N_0^+ \).
We split \( \Xi_1(N) \) into three parts as follows,
\[ \Xi_1(N) = \sum_{t_j \leq Z} (\ldots) + \sum_{t_j > Z} (\ldots) + \sum_{ t_j \text{ exc.} } (\ldots). \]
The sum over the exceptional eigenvalues causes no problems in this case, as the respective Bessel transforms are very small.
The rest can be treated in the same way as above, and we get the bounds
\begin{align}
  \int \! G_0(\lambda) \Xi_1(N) \, d\lambda &\ll {v^\circ}^\frac12 v (r_2 x_1)^{\frac12 + \varepsilon} \frac{ h^\theta }{\Omega^\frac12}, \label{eqn: fourth preliminary estimate} \\
  \int \! G_0(\lambda) \Xi_1(N) \, d\lambda &\ll {v^\circ}^\frac12 v (r_2 x_1)^{\frac12 + \varepsilon} \frac1{ \Omega^\frac12 } \left( 1 + \Omega^\frac12 \frac{ (r_1 r_2 v , h )^\frac14 h^\frac14 }{ (r_1 r_2)^\frac12 v^\frac14} \right). \label{eqn: fifth preliminary estimate}
\end{align}

The same reasoning applies similarly to \( \Xi_2(N) \) and \( \Xi_3(N) \), the main difference being that we don't have to worry about exceptional eigenvalues at all.
In the end we get from \eqref{eqn: first preliminary estimate} and \eqref{eqn: fourth preliminary estimate},
\begin{align*}
  R_{AB}^+(N; \chi) &\ll {v^\circ}^\frac12 v (r_2 x_1)^{\frac12 + \varepsilon} \left( \frac{ h^\theta }{\Omega^\frac12} + (r_2 x_1)^\theta \right), \\
  \intertext{from \eqref{eqn: third preliminary estimate} and \eqref{eqn: fifth preliminary estimate},}
  R_{AB}^+(N; \chi) &\ll {v^\circ}^\frac12 v (r_2 x_1)^{\frac12 + \varepsilon} \left( \frac1{\Omega^\frac12} + \left( \frac{ x_1 (r_1 r_2 v, h) }{ {r_1}^2 r_2 v } \right)^\theta \left( 1 + \frac{ (r_1 r_2 v , h )^\frac14 h^\frac14 }{ (r_1 r_2)^\frac12 v^\frac14} \right) \right),
  \intertext{and from \eqref{eqn: second preliminary estimate} and \eqref{eqn: fifth preliminary estimate},}
  R_{AB}^+(N; \chi) &\ll {v^\circ}^\frac12 v (r_2 x_1)^{\frac12 + \varepsilon} \left( \frac1{\Omega^\frac12} + \left( \frac{r_2 x_1}h \right)^\theta \left( 1 + \frac{ (r_1 r_2 v , h )^\frac14 h^\frac14 }{ (r_1 r_2)^\frac12 v^\frac14} \right) \right).
\end{align*}
Taking account of \eqref{eqn: sum over conductors}, these bounds eventually lead to \eqref{eqn: first main estimate}, \eqref{eqn: second main estimate} and \eqref{eqn: third main estimate}.

\subsection{The main term} \label{subsection: The main term}

The only thing left to do is the evaluation of the main term.
After summing over all \(A\) and \(B\), it has the form
\begin{align}
  \Sigma^0 &:= \frac1{r_1} \sum_{u_2^\ast \mid u_2} \sum_\twoln{a}{ (a, s_2 u_2^\ast) = 1 } \frac1a \int \! \lambda_{ f_1 - r_1 g_2 \overline{s_1}, r_1 a }(\xi) f(\xi; a) \, d\xi \nonumber \\
    &= \int \! w_1\left( \frac{r_1 \xi + f_1}{x_1} \right) w_2\left( \frac{r_2 \xi + f_2}{x_2} \right) \left( \sum_{u_2^\ast \mid u_2} \tilde \Sigma^0(\xi, u_2^\ast) \right) \, d\xi, \label{eqn: definition of Sigma^0}
\end{align}
with
\begin{align}
  \tilde \Sigma^0(\xi, u_2^\ast) &:= \sum_\twoln{a}{ (a, s_2 u_2^\ast) = 1 } \frac{ \lambda_{ f_1 - r_1 g_2 \overline{s_2}, r_1 a }(r_1 \xi + f_1) }a h\left( \frac{u_2 a}{u_2^\ast}, \frac{u_2^\ast}{u_2 a} (r_2 \xi + f_2) \right) \nonumber \\
    &= \frac1{2\pi i} \int_{ (\sigma) } \! \hat h(s; \xi) Z(s; \xi) \, ds, \label{eqn: tilde Sigma as integral}
\end{align}
where \( \hat h(s; \xi) \) is the Mellin transform
\[ \hat h(s; \xi) := \int_0^\infty \! h\left( \frac{u_2 a}{u_2^\ast}, \frac{u_2^\ast}{u_2 a} (r_2 \xi + f_2) \right) a^{s - 1} \, da, \quad \Re(s) > 0, \]
and the function \( Z(s; \xi) \) is defined as the Dirichlet series
\[ Z(s; \xi) := \sum_\twoln{a}{ (a, s_2 u_2^\ast) = 1 } \frac{ \lambda_{ f_1 - r_1 g_2 \overline{s_2}, r_1 a }(r_1 \xi + f_1) }{ a^{1 + s} }, \quad \Re(s) > 0. \]
The integral in \eqref{eqn: tilde Sigma as integral} is initially defined for \( \sigma > 0 \).
Our plan is to move the line of integration to \( \sigma = -1 + \varepsilon \), so that we can use the residue theorem to extract a main term.
A meromorphic continuation of \( \hat h(s; \xi) \) can easily be found by using partial integration.
For \( Z(s; \xi) \) the situation is not quite as obvious.

Define the operator
\[ \Delta_\delta(\xi) := \left( \log\xi + 2\gamma + 2 \frac\partial{\partial \delta} \right) \bigg |_{\delta = 0}, \]
so that we can write
\[ \lambda_{ f_1 - r_1 g_2 \overline{s_2}, r_1 a }(r_1 \xi + f_1) = \Delta_\delta(r_1 \xi + f_1) \sum_{d \mid r_1 a} \frac{ c_d( f_1 - r_1 g_2 \overline{s_2} ) }{ d^{1 + \delta} }. \]
Now we separate the part of \(r_1\) which shares common factors with \( s_2 u_2^\ast\) from the rest by setting
\[ v := \left( r_1, (s_2 u_2^\ast)^\infty \right), \quad t_1 := \frac{r_1}v, \]
so that
\[ \sum_{d \mid r_1 a} \frac{ c_d( f_1 - r_1 g_2 \overline{s_2} ) }{ d^{1 + \delta} } = \left( \vcent{ \sum_{d \mid v} \frac{ c_d(f_1) }{ d^{1 + \delta} } } \right) \left( \vcent{ \sum_{d \mid t_1 a} \frac{ c_d(h_0 u_1) }{ d^{1 + \delta} } } \right), \]
and hence
\[ Z(s; \xi) = \Delta_\delta(r_1 \xi + f_1) \left( \vcent{ \sum_{d \mid v} \frac{ c_d(f_1) }{ d^{1 + \delta} } } \right) \sum_\twoln{a}{ (a, s_2 u_2^\ast) = 1 } \frac1{ a^{1 + s} } \sum_{d \mid t_1 a} \frac{ c_d(h_0 u_1) }{ d^{1 + \delta} }. \]
The two outer sums can be transformed to
\[ \sum_\twoln{a}{ (a, s_2 u_2^\ast) = 1 } \frac1{ a^{1 + s} } \sum_{d \mid t_1 a} \frac{ c_d(h_0 u_1) }{ d^{1 + \delta} } = \sum_\twoln{d}{ (d, s_2 u_2^\ast) = 1 } \frac{ c_d(h_0 u_1) (d, t_1) }{ d^{2 + \delta} } \tilde Z(s; d), \]
with
\[ \tilde Z(s; d) := \zeta(1 + s) \frac{ (d, t_1)^s }{ d^s } \prod_{ p \mid s_2 u_2^\ast } \left( 1 - \frac1{ p^{1 + s} } \right). \]
This is a meromorphic function, defined on the whole complex plane, which means that the desired meromorphic continuation for \( Z(s; \xi) \) can be given by
\[ Z(s; \xi) = \Delta_\delta(r_1 \xi + f_1) \left( \vcent{ \sum_{d \mid v} \frac{ c_d(f_1) }{ d^{1 + \delta} } } \right) \left( \vcent{ \sum_\twoln{d}{ (d, s_2 u_2^\ast) = 1 } \frac{ c_d(h_0 u_1) (d, t_1) }{ d^{2 + \delta} } \tilde Z(s; d) } \right). \]
Hence
\[ \tilde \Sigma^0(\xi, u_2^\ast) = \Delta_\delta(r_1 \xi + f_1) \left( \vcent{ \sum_{d \mid v} \frac{ c_d(f_1) }{ d^{1 + \delta} } } \right) \left( \vcent{ \sum_\twoln{d}{ (d, s_2 u_2^\ast) = 1 } \frac{ c_d(h_0 u_1) (d, t_1) }{ d^{2 + \delta} } \tilde I^0(\xi, d) } \right), \]
with
\[ \tilde I^0(\xi, d) := \frac1{2\pi i} \int_{ (\sigma) } \! \hat h(s; \xi) \tilde Z(s; d) \, ds. \]

The Mellin transform \( \hat h(s; \xi) \) has at \( s = 0 \) the Taylor expansion
\[ \hat h(s; \xi) = \frac2s + \log(r_2 \xi + f_2) + 2 \log\frac{u_2^\ast}{u_2} + \BigO{s}, \]
while that of \( \tilde Z(s; d) \) is given by
\[ \tilde Z(s; d) = \left( \frac1s + \gamma + \frac\partial{\partial \rho} \right) \bigg|_{\rho = 0} \frac{ (d, t_1)^\rho }{ d^\rho } \prod_{ p \mid s_2 u_2^\ast } \left( 1 - \frac1{ p^{1 + \rho} } \right) + \BigO{s}. \]
All in all, the residue of their product at \( s = 0 \) is
\[ \Res{s = 0} \left( \hat h(s; \xi) \tilde Z(s; d) \right) = \Delta_\rho(r_2 \xi + f_2) \left( \frac{u_2^\ast}{u_2} \right)^\rho \frac{ (d, t_1)^\rho }{ d^\rho } \prod_{ p \mid s_2 u_2^\ast } \left( 1 - \frac1{ p^{1 + \rho} } \right). \]
We now move the line of integration to \( \sigma = -1 + \varepsilon \),
\[ \tilde I^0(\xi, d) = \Delta_\rho(r_2 \xi + f_2) \left( \frac{u_2^\ast}{u_2} \right)^\rho \frac{ (d, t_1)^\rho }{ d^\rho } \prod_{ p \mid s_2 u_2^\ast } \left( 1 - \frac1{ p^{1 + \rho} } \right) + \BigO{ \frac{ d^{1 - \varepsilon} }{ {x_2}^{\frac12 - \varepsilon} } }, \]
and hence
\[ \tilde \Sigma^0(\xi, u_2^\ast) = \Delta_\delta(r_1 \xi + f_1) \Delta_\rho(r_2 \xi + f_2) \tilde M_{\delta, \rho}^0(\xi, u_2^\ast) + \BigO{ \frac{ {x_2}^\varepsilon }{ {x_2}^\frac12 } }, \]
with
\begin{multline*}
  \tilde M_{\delta, \rho}^0(\xi, u_2^\ast) := \left( \frac{u_2^\ast}{u_2} \right)^\rho \left( \vcent{ \sum_{d \mid v} \frac{ c_d(f_1) }{ d^{1 + \delta} } } \right) \prod_{ p \mid s_2 u_2^\ast } \left( 1 - \frac1{ p^{1 + \rho} } \right) \\
    \cdot \left( \vcent{ \sum_\twoln{d}{ (d, s_2 u_2^\ast) = 1 } \frac{ c_d(h_0 u_1) (d, t_1)^{1 + \rho} }{ d^{2 + \delta + \rho} } } \right).
\end{multline*}
An elementary but quite tedious calculation shows that this product can be transformed in such a way that
\[ \sum_{ u_2^\ast \mid u_2 } \tilde M_{\delta, \rho}^0(\xi, u_2^\ast) = C_{\delta, \rho}(r_1, r_2, f_1, f_2), \]
where
\[ C_{\delta, \rho}(r_1, r_2, f_1, f_2) := \sum_\twoln{ u_1^\ast \mid u_1 }{ u_2^\ast \mid u_2 } \left( \frac{u_1^\ast}{u_1} \right)^\delta \left( \frac{u_2^\ast}{u_2} \right)^\rho \psi_\delta(s_1 u_1^\ast) \psi_\rho(s_2 u_2^\ast) \gamma_{\delta + \rho}(s_1 u_1^\ast s_2 u_2^\ast), \]
with
\[ \psi_\alpha(n) := \prod_{p \mid n} \left( 1 - \frac1{ p^{1 + \alpha} } \right) \quad \text{and} \quad \gamma_\alpha(n) := \sum_{ (d, n) = 1 } \frac{ c_d(h_0) }{ d^{2 + \alpha} }. \]
After a look back at \eqref{eqn: definition of Sigma^0}, we see that our main term has the form
\[ \Sigma^0 = M(x_1, x_2) + \BigO{ \frac{ {x_2}^{\frac12 + \varepsilon} }{r_2} }, \]
with
\[ M(x_1, x_2) := \int \! w_1\left( \frac{r_1 \xi + f_1}{x_1} \right) w_2\left( \frac{r_2 \xi + f_2}{x_2} \right) P( \log(r_1 \xi + f_1), \log(r_2 \xi + f_2) ) \, d\xi, \]
where \( P(\xi_1, \xi_2) \) is the quadratic polynomial given by
\begin{align} \label{eqn: the polynomial in the main term}
  P(\log \xi_1, \log \xi_2) := \Delta_\delta(\xi_1) \Delta_\rho(\xi_2) C_{\delta, \rho}(r_1, r_2, f_1, f_2).
\end{align}
This concludes the proof of \eqref{eqn: main term}.

\bibliography{On_a_certain_additive_divisor_problem}

\end{document}